\documentclass[reqno]{amsart}

\usepackage[margin= 1.0 in]{geometry}
\usepackage{amsfonts}
\usepackage{amsthm}
\usepackage{amsmath}
\usepackage{mathrsfs}
\usepackage{graphicx}
\usepackage{amssymb}
\usepackage{xtab}
\usepackage{tikz-cd} 
\usepackage{ytableau}
\usepackage{xfrac}
\usepackage{stackrel}
\usepackage{color}
\usepackage{enumerate}
\usepackage{mathtools}
\usepackage{hyperref}
\hypersetup{
   colorlinks=true,
   citecolor=blue,
   linkcolor=violet
   }

\usepackage[normalem]{ulem}

\tikzcdset{scale cd/.style={every label/.append style={scale=#1},
    cells={nodes={scale=#1}}}}
\DeclareMathOperator{\GW}{\mathbb{G}W} 
\newcommand{\Gr}{\operatorname{Gr}}

\newtheorem{theorem}{Theorem}[section]
\newtheorem{corollary}[theorem]{Corollary}
\newtheorem{proposition}[theorem]{Proposition}
\newtheorem{lemma}[theorem]{Lemma}

\theoremstyle{definition}
\newtheorem{remark}[theorem]{Remark}

\newtheorem{definition}[theorem]{Definition}

\usepackage{cleveref}

\makeatletter
\let\@wraptoccontribs\wraptoccontribs
\makeatother

\usepackage[all]{xy}

\newcommand{\D}{\mathcal{D}}

\newcommand{\Hom}{\operatorname{Hom}}
\newcommand{\Z}{\mathbb{Z}}
\newcommand{\Spec}{\operatorname{Spec}}

\newcommand{\A}{\mathcal{A}}

\newcommand{\T}{\mathcal{T}}
\newcommand{\quis}{\operatorname{quis}}

\newcommand{\B}{\mathcal{B}}
\newcommand{\sPerf}{\operatorname{sPerf}}

\newcommand{\fgt}{\operatorname{fgt}}
\newcommand{\I}{\mathcal{I}}
\newcommand{\can}{\operatorname{can}}

\newcommand{\E}{\mathcal{E}}
\newcommand{\pp}{\mathbb{P}}
\newcommand{\Kosz}{\operatorname{Koz}}
\newcommand{\cone}{\operatorname{cone}}
\newcommand{\Qcoh}{\operatorname{Qcoh}}
\newcommand{\calK}{\mathcal{K}}
\newcommand{\calH}{\mathcal{H}}
\newcommand{\calL}{\mathcal{L}}
\newcommand{\K}{\mathbb{K}}

\makeatletter
\renewcommand\subsection{%
  \@startsection{subsection}{2}{\z@}%
    {3.0ex \@plus 1ex \@minus .2ex}
    {1.5ex \@plus .2ex}
    {\normalfont\bfseries}
}
\makeatother

\begin{document}

\title{On the Hermitian K-theory of Grassmannians}

\begin{abstract}
We compute the $\mathbb{G}W$-spectrum (Karoubi--Grothendieck--Witt spectrum) of Grassmannians over divisorial schemes defined over fields of characteristic zero, and, as a corollary, determine their stabilized $\mathbb{L}$-theory spectrum. As part of our computation, we establish split fibration sequences in $\mathbb{K}$-theory, $\mathbb{G}W$-theory and $\mathbb{L}$-theory.  
\end{abstract}

\author{Sunny Sood and Chunkai Xu, \\ \vspace{-1ex} \\ with an appendix by Marco Schlichting}

\address{Mathematics Institute \\
 University of Warwick \\
 CV4 7AL \\
 United Kingdom}

\email{Sunny.Sood.1@warwick.ac.uk}

\address{Mathematics Institute \\
 University of Warwick \\
 CV4 7AL \\
 United Kingdom}

\email{Chunkai.Xu@warwick.ac.uk}

\address{Mathematics Institute \\
 University of Warwick \\
 CV4 7AL \\
 United Kingdom}

\email{m.schlichting@warwick.ac.uk}

\keywords{Hermitian K-theory, Grassmannians}
\subjclass[2020]{14M15, 19G38}

\maketitle
\tableofcontents

\section{Introduction}


Hermitian K-theory, which may be thought of as the K-theory of quadratic forms, has been the subject of significant attention over recent years. As an invariant of algebraic varieties, Hermitian K-theory has become an important invariant in \textit{motivic homotopy theory}. For example, Hermitian K-theory has been used to compute the endomorphism ring of the motivic sphere spectrum \cite{morel2012a1} (see also \cite{rondigs2019first}) and was an important tool in the Asok--Fasel vector bundle classification program \cite{MR3273577}. Its connection to K-theory and L-theory via the fundamental fibre sequence (see  \cite[Theorem 8.13]{schlichting17}, \cite{calmes2025hermitian}) also enriches the study of the subject, as these theories have deep and diverse applications throughout mathematics. We refer the reader to \cite{marlowe2024higher} and \cite{calmes2023hermitian} for a summary of the subject's history and a survey of current applications. However, explicit computations of Hermitian K-theory remain relatively limited. Some examples include projective spaces \cite{karoubi2021grothendieck}, projective bundles over divisorial schemes \cite{rohrbach2022projective}, and Grassmannians over a regular base \cite{huang2023connecting}. 

In this paper, we present a computation for the Hermitian K-theory of Grassmannians over divisorial schemes defined over fields of characteristic zero. Our proof uses the generalized additivity theorem for Hermitian K-theory, first introduced in \cite{karoubi2021grothendieck}, the full exceptional collection on the derived category of Grassmannians constructed by Kapranov in \cite{kapranov1988derived} (and since generalised by \cite{buchweitz2015derived} and \cite{efimov2017derived}), and the base change for semi-orthogonal decomposition by \cite{kuznetsov_base_2011}. This computation generalises the result for projective space computed in \cite{rohrbach2022projective} and \cite{karoubi2021grothendieck}, and coincides with the computation obtained in \cite{huang2023connecting}, which was done for Grassmannians over a regular base. 

As a corollary, we also compute the stabilized $\mathbb{L}$-theory spectrum of Grassmannians and verify that the corresponding Witt groups coincide with the computation obtained in \cite{balmer2012witt}, which was done for Grassmannians over a regular, noetherian and separated scheme over $\mathbb{Z}[\frac{1}{2}]$ of finite Krull dimension. As part of our computation, we establish split fibration sequences in $\mathbb{K}$-theory, $\mathbb{G}W$-theory and $\mathbb{L}$-theory.

Our computation is summarised by the following two results: 
\begin{theorem}[See \Cref{thm:compute} below]
Let $S$ be a divisorial scheme defined over a field of characteristic zero. Let $k \leq n$ and let $\Gr^{k,n}_{S}$ denote the Grassmannian over $S$. Let $\mathcal{Q}$ denote the tautological bundle of rank $k$ on $\Gr^{k,n}_{S}$. 
Then, we have an equivalence of spectra 
\[ 
\mathbb{G}W^{[r]}(\Gr^{k,n}_{S},(\det\mathcal{Q})^{\otimes (n-k)}) \cong 
\begin{cases}
\mathbb{K}(S)^{\oplus \frac{1}{2}R(k,n-k)} & k(n-k) \,\, \text{odd,} \\
\mathbb{G}W^{[r - 2]}(S)^{\oplus S(k,n-k)} \oplus \mathbb{K}(S)^{\oplus \frac{1}{2}A(k,n-k)} & (n-k) \,\,  \text{odd and } k\equiv 2 (\text{mod }4),\\
\mathbb{G}W^{[r]}(S)^{\oplus S(k,n-k)} \oplus \mathbb{K}(S)^{\oplus \frac{1}{2}A(k,n-k)} &  \text{otherwise,} \\
\end{cases}
 \]
   where $R(k,l)$ denotes the number of Young diagrams of height at most $k$ and width at most $l$, $S(k,l)$ denotes the number of symmetric partitions in such a collection and $A(k,l)$ denotes the number of asymmetric partitions in such a collection.
\end{theorem}

In \Cref{sec:prelim} we will define the notion of symmetric and asymmetric partitions. We will compute $R(k,l) = \binom{k+l}{k}$, $S(k,l) = \binom{k'+l'}{k'}$, where $k'$ and $l'$ are the integer parts of $k/2$ and $l/2$, respectively, and $A(k,l) = R(k,l) - S(k,l)$.

The computation is then completed by proving the following splitting result.
\begin{theorem}[See Theorem \ref{thm:split} below]
Let $i: \Gr_{S}^{k-1,n-1} \rightarrow \Gr_{S}^{k,n}$, $j: \Gr_{S}^{k,n-1} \rightarrow \Gr_{S}^{k,n}$ be the canonical regular embeddings of Grassmannians. Then, in the above notation, we have the following split fibration of $\mathbb{G}W$-spectra for all $k\leq n, r \in \mathbb{Z}$:
\[
\mathbb{G}W^{[r-(n-k)]}(\Gr^{k-1,n-1}_{S},(\det\mathcal{Q})^{\otimes (n-k)}) \xrightarrow{i_*}\mathbb{G}W^{[r]}(\Gr^{k,n}_{S},(\det\mathcal{Q})^{\otimes (n-k-1)}) \xrightarrow{j^*} \mathbb{G}W^{[r]}(\Gr^{k,n-1}_{S},(\det\mathcal{Q})^{\otimes (n-k-1)}).
\]
\end{theorem}
The work needed to prove this theorem is actually sufficient to establish a similar split fibration sequence in $\mathbb{K}$-theory. Whilst this result is known for Grassmannians over fields of characteristic zero, at our level of generality, this appears to be new.  
\begin{corollary}[See \Cref{cor:Ksplit}]
In the above notation, we have the following split fibration of $\mathbb{K}$-theory spectra for all $k\leq n, r \in \mathbb{Z}$:
\[
\mathbb{K}(\Gr^{k-1,n-1}_{S}) \xrightarrow{i_{*}}\mathbb{K}(\Gr^{k,n}_{S})\xrightarrow{j^{*}} \mathbb{K}(\Gr^{k,n-1}_{S}).
\]
\end{corollary}

Moreover, considering the fundamental fibre sequence involving $\mathbb{L}$-theory (see \cite[Theorem 8.13]{schlichting17} or \cite{calmes2025hermitian}) 
\begin{align} \label{ltheorysequence}
\mathbb{K}^{[r]}_{hC_{2}} \rightarrow \mathbb{G}W^{[r]} \rightarrow \mathbb{L}^{[r]},
\end{align}
we obtain a computation for the $\mathbb{L}$-theory of Grassmannians:
\begin{corollary}[see \Cref{cor:Ltheory} below]
    In the above notation, we have 
    \[
\mathbb{L}^{[r]}(\Gr^{k,n}_{S},(\det\mathcal{Q})^{\otimes (n-k)}) \cong 
\begin{cases}
0 & k(n-k) \,\, \text{odd,} \\
\mathbb{L}^{[r - 2]}(S)^{\oplus S(k,n-k)} & (n-k) \,\,  \text{odd and } k\equiv 2 (\text{mod }4),\\
\mathbb{L}^{[r]}(S)^{\oplus S(k,n-k)}  &  \text{otherwise,} \\
\end{cases}
    \]
 
 and split fibrations of stabilized $\mathbb{L}$-theory spectra
\[
\mathbb{L}^{[r-(n-k)]}(\Gr^{k-1,n-1}_{S},(\det\mathcal{Q})^{\otimes (n-k)}) \xrightarrow{i_*}\mathbb{L}^{[r]}(\Gr^{k,n}_{S},(\det\mathcal{Q})^{\otimes (n-k-1)}) \xrightarrow{j^*} \mathbb{L}^{[r]}(\Gr^{k,n-1}_{S},(\det\mathcal{Q})^{\otimes (n-k-1)}).
\]
\end{corollary}

\begin{remark}
    In Appendix \ref{App:ExplicitComparison}, using all theorems above, we give closed formulae for $\GW$-, $\mathbb{L}$-theory of Grassmannians. We also give explicit comparison of our results with known results in \cite{balmer2012witt,rohrbach2021atiyah,rohrbach2022projective,huang2023connecting,karoubi2021grothendieck}.
\end{remark}

\begin{remark}
Before our work, a partial computation of Hermitian K-theory had been obtained in \cite{rohrbach2021atiyah}. They compute the Hermitian K-theory of the so-called \textit{even} Grassmannians using a different semi-orthogonal decomposition, but the odd-dimensional Grassmannians were not taken into account. Computations for projective space/projective bundle have been given by \cite{karoubi2021grothendieck} and \cite{rohrbach2022projective}.

Another computation of Hermitian K-theory of Grassmannian has been given in \cite[Theorem 10.39]{huang2023connecting}. They use a blow-up approach analogous to \cite{balmer2012witt}, while our approach relies on the exceptional collection constructed by Kapranov \cite{kapranov1988derived}. Our results hold for divisorial schemes defined over a field of characteristic zero, whereas the methods in \cite{huang2023connecting} allow the base to be any regular scheme. Their computation is described in terms of \textit{even} and \textit{buffalo-check} Young diagrams, whereas our computations are described in terms of \textit{asymmetric} and \textit{symmetric} partitions. 

The Witt groups of Grassmannians were computed in \cite{balmer2012witt}, and preceded Hermitian K-theory computations of Grassmannians. The case for $S = \mathbb{C}$ was computed separately by \cite{zibrowius2011witt}. The computation in \cite{balmer2012witt} relies on their companion articles \cite{balmer2009geometric} and \cite{balmer2012bases}. On the other hand, we compute the stabilized $\mathbb{L}$-theory spectrum of Grassmannians as a \textit{corollary} of our Hermitian K-theory computation, using the fundamental fibre sequence (\ref{ltheorysequence}). 
\end{remark}

\paragraph{\textit{Structure of the paper}}
 In Section \ref{sec:prelim} we recall some notions about Young diagrams and Kapranov's full exceptional collection on Grassmannians, then we prove two technical lemmas which will be used in Section \ref{sec:computation}. In Section \ref{sec:computation} we compute the $\mathbb{G}W$-spectrum of Grassmannians, and obtain $\mathbb{L}$-theory as a corollary. In Section \ref{sec:mutation}, we compute several mutations using Borel--Weil--Bott theorem, which are used in Section \ref{sec:computation}. In Appendix \ref{appendix:marco}, written by Marco Schlichting, the push-forward of regular embeddings in Hermitian K-theory is developed. In Appendix \ref{App:ExplicitComparison} we give closed formulae for $\GW$-, $\mathbb{L}$-theory of Grassmannians, and explicitly compare our computation with known results.

\section*{Acknowledgements}
The authors would like to express their gratitude to Marco Schlichting, Daniel Marlowe, Martin Gallauer, Chunyi Li, Zhiyu Liu, Xiangdong Wu and Weiyi Zhang for insightful discussions. Special thanks are due to Heng Xie for sharing an instructive note. We particularly thank Anton Fonarev for identifying an error in the initial draft of this manuscript and for his generous email responses. We thank the anonymous referee for a careful reading as well as a list of useful suggestions that improved the exposition of the paper. The first author would particularly like to thank Marco Schlichting for all his help over the years - this paper would not exist without him. The first author also thanks the online math community for helping him with his basic questions surrounding this problem. The second author would particularly like to thank Chunyi Li for his kind and patient guidance.  The second author is supported by the Warwick Mathematics Institute Centre for Doctoral Training, and gratefully acknowledges funding from the University of Warwick.

\section{Preliminaries} \label{sec:prelim}

In this section, we collect some definitions, known results and technical lemmas.

\subsection{Some combinatorial notions}
Denote by $Y(k,l)$ the set of all Young diagrams of height at most $k$ and width at most $l$, which can be identified with the set of sequences of integers:
\begin{equation*}
    Y(k,l) = \{\lambda\in \mathbb{Z}^k| l\geq \lambda_1\geq \lambda_2\geq\cdots\geq\lambda_k\geq 0 \}.
\end{equation*}

We will write $|\lambda| \coloneqq \sum_{i=1}^{k}\lambda_{i}$ for the \textit{modulus} of the Young diagram $\lambda$ and $R(k,l) \coloneqq \# Y(k,l)$. Note that for $k\geq 2$, we have formula $R(k,l) = \sum_{i=0}^{l} R(k-1,i)$, giving us a recursion relation $R(k,l) = R(k,l-1) + R(k-1,l)$. Setting $R(k,0) = R(0,l) =1$, we compute that $R(k,l) = \binom{k+l}{k}$. In particular, note that $R(k,l) = R(l,k)$. 

In addition, for $\lambda \in Y(k,l)$, let $\lambda^{c}$ denote the Young diagram given by the \textit{complement} of $\lambda$ in the $k\times l$ rectangle through $180^\circ$ . In formulae, if $\lambda = (\lambda_{1},\dots, \lambda_{k})$, then $\lambda^{c} = (l-\lambda_{k},\dots, l- \lambda_{1})$. For a visual description, see Figure \ref{figure:youngdiagram}, which shows two Young diagrams that are complementary to each other within a given rectangle. Note that $(\lambda^{c})^{c} = \lambda$.
\begin{figure}[h]
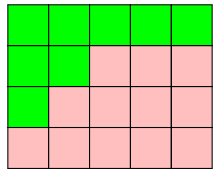

  \begin{ytableau} 
   *(green) & *(green)  &*(green) & *(green)  &*(green)  \\
   *(green)  & *(green) & *(pink) & *(pink) & *(pink) \\
   *(green) & *(pink) & *(pink) & *(pink) & *(pink) \\
   *(pink) & *(pink) & *(pink) & *(pink) & *(pink)
  \end{ytableau}
  \caption{Illustrating two Young diagrams which are complementary to each other within a given rectangle.} \label{figure:youngdiagram}

\par\smallskip
{\small\noindent \textbf{Alt text:} A four-row, five-column rectangle contains a green Young diagram in the upper left and a pink complementary region in the lower right. Rotating the pink region by 180 degrees produces the complementary Young diagram. \par}
\end{figure}

Following \cite{rohrbach2021atiyah}, we call a Young diagram $\lambda$ a \textit{symmetric partition} if $\lambda^{c} = \lambda$. Note that $|\lambda^{c}| = kl - |\lambda|$, so that symmetric partitions do not exist when $kl$ is odd.  

We denote the number of symmetric partitions in $Y(k,l)$ by $S(k,l)$. In particular, when $kl$ is odd $S(k, l ) = 0$.  In \cite{rohrbach2021atiyah} it was argued that the number $S(k,l)$ is equivalent to the number of \textit{palindrome} sequences with $k$ zeros and $l$ ones. 

Indeed, a Young diagram may be traced out by a path starting at the bottom left of the given rectangle and finishing at the top right of the rectangle, moving $l$ steps to the right and $k$ steps upwards. The right steps correspond to the digit 1 and the upward steps correspond to the digit 0. For example, the green Young diagram above would have sequence $(0,1,0,1,0,1,1,1,0)$. The symmetric Young diagrams then precisely correspond to the \textit{palindrome sequences}. In particular, defining the \textit{dual} of a binary sequence to be the binary sequence obtained by swapping each bit (e.g. 010 becomes 101), we deduce that $S(k,l) = S(l,k)$, as a binary sequence is palindrome if and only if its dual is palindrome.
Therefore, whenever $k$ or $l$ is even, we compute that $S(k,l) = \binom{k'+l'}{k'}$, where $k'$ and $l'$ are the integer parts of $k/2$ and $l/2$ respectively. Otherwise, this number is zero. In particular, this description gives us the identity $S(k,l) = S(k,l-1)$ when $k$ is even, $l$ is odd and at least $1$, and the identity $S(k-1,l) + S(k,l-1) = S(k,l)$ when $k$ and $l$ are both even and at least 1.

A Young diagram which is not a symmetric partition will be called an \textit{asymmetric partition}. We denote the number of asymmetric partitions in $Y(k,l)$ by $A(k,l)$. Note that $R(k,l) = S(k,l) + A(k,l)$.

For any sequence of non-increasing integers $\lambda = (\lambda_1,\cdots,\lambda_k)$, we denote by
\begin{equation*}
    \overline{\lambda}\coloneqq (\lambda_1 - \lambda_k,\lambda_2 - \lambda_k,\cdots,\lambda_{k-1}-\lambda_k,0),
\end{equation*}
the `minimal positive' modification of $\lambda$. We will also write $-\lambda \coloneqq (-\lambda_k, \cdots, -\lambda_1)$. For $r \in \mathbb{Z}, k \in \mathbb{N}$, we will write $(r)^{k} \coloneqq (r,\cdots,r)$, $k$ entries of $r$. This notation will be used in Lemma \ref{finalsteplemma} and Section \ref{sec:mutation}.

\subsection{Derived categories of Grassmannians, Schur functors }
When $S$ is a divisorial scheme defined over a field $\mathbb{F}$ of characteristic $0$, denote by $\Gr^{k,n}_{S}$ the Grassmannian over $S$, where $k \leq n$. Note that, by construction, we have a pullback diagram 
\begin{center}
\begin{tikzcd}
\Gr^{k,n}_{S} \arrow[r, "\phi"] \arrow[d, "p"] 
& \Gr^{k,n}_{\mathbb{F}} \arrow[d, "p"] \\
S \arrow[r, "\phi"] & \text{Spec}(\mathbb{F}) 
\end{tikzcd}.
\end{center} 
On $\Gr^{k,n}_{S}$, we have the dual tautological exact sequence of vector bundles 
\begin{align} \label{s.e.s:universal}
0 \rightarrow \mathcal{R}_{k,n} \rightarrow \mathcal{O}^{\oplus n}_{\Gr^{k,n}_{S}} \rightarrow \mathcal{Q}_{k,n} \rightarrow 0,
\end{align}
where $\mathcal{R} = \mathcal{R}_{k,n}$ is of rank $n-k$ and $\mathcal{Q} = \mathcal{Q}_{k,n}$ is of rank $k$. These are sometimes called the dual bundle and tautological bundle respectively. We adopt this notion because it is convenient for computation. When clear from context, we will omit the decorations put on $\mathcal{R}$ and $\mathcal{Q}$. Given a rank $r$ vector bundle $\mathcal{F}$, we will write $\det \mathcal{F} \coloneqq \bigwedge^{r}\mathcal{F}$ for its top exterior power.

Given a Young diagram $\lambda = (\lambda_1,\cdots,\lambda_k)\in Y(k, n-k)$, we denote its transpose by $\lambda^t = (\lambda^t_1,\cdots,\lambda^t_{n-k})$.

Furthermore, for a Young diagram $\lambda = (\lambda_1,\cdots,\lambda_k)$, we use the following notation:
\begin{equation*}
    \Box^\lambda \mathcal{Q} \coloneqq \Box^{\lambda_1} \mathcal{Q} \otimes \Box^{\lambda_2}\mathcal{Q} \otimes \cdots \otimes \Box^{\lambda_k} \mathcal{Q},
\end{equation*}
where the box $\Box$ can be taken as wedge product $\wedge$, tensor power $T$, symmetric power $\text{Sym}$.

To characterize the (skew-)symmetry of the canonical pairing in the next subsection, we use the following definition of Schur functor:
\begin{definition}[{\cite[Definition II.1.3]{AKIN1982207}}]
    The \textit{Schur functor }$S^\lambda\mathcal{Q}$ of $\mathcal{Q}$ with respect to $\lambda$ with modulus $d\coloneqq |\lambda|$ is defined to be the image of the following composition
    \begin{equation*}
        \wedge^{\lambda^t} \mathcal{Q}\rightarrow \mathcal{Q}^{\otimes d} \xrightarrow{s_{\lambda^{t}}} \mathcal{Q}^{\otimes d} \rightarrow \text{Sym}^\lambda \mathcal{Q}. 
    \end{equation*}
    Here, the first map is the inclusion given by
    \begin{equation*}
        \wedge^i \mathcal{Q} \rightarrow T^{i}\mathcal{Q}:  x_1\wedge\cdots\wedge x_i \mapsto \sum\limits_{\sigma\in\Sigma_i} \text{Sgn}(\sigma) \cdot x_{\sigma(1)} \otimes \cdots \otimes x_{\sigma(i)}, 
    \end{equation*}
    the second map is defined via $s_{\lambda^{t}}(v_{1}\otimes \cdots \otimes v_{d}) = v_{\sigma_{\lambda^{t}}(1)} \otimes \cdots\otimes v_{\sigma_{\lambda^{t}}(d)}$, where $\sigma_{\lambda^{t}}(\lambda_{1}^{t}+\cdots + \lambda_{i-1}^{t} + j) \coloneqq \lambda_{1}+\cdots + \lambda_{j-1} + i$ for all $1\leq j \leq \lambda_{i}^{t}$; and the third map is the canonical quotient map. 
\end{definition}
\begin{remark}
    The permutation $\sigma_{\lambda} \in S_{d}$ may be thought of as follows: any integer $r \in \{1,\dots,d \}$ can be uniquely written as the sum $r = \lambda_{1} + \cdots + \lambda_{i-1} + j $ for $1\leq j \leq \lambda_{i}$. Pictorially, the pair $(i,j)$ describes the position ($i$th row and $j$th column) of $r$ in Young diagram $\lambda$. The permutation $\sigma_{\lambda}$ then maps $r$ to the integer $r^{t}\in \{1,\dots,d \}$ corresponding to the $(j,i)$ position in the Young diagram $\lambda^{t}$.
\end{remark}

 For a vector bundle $\mathcal{F}$ on $\Gr^{k,n}_{S}$, the above convention yields $S^{(m)}\mathcal{F} = \text{Sym}^{m} \mathcal{F}, S^{(1)^m} \mathcal{F}=  \Lambda^m\mathcal{F}$.

 In the following proposition, $\text{Vect}(\text{Gr}^{k,n}_{S})$ refers to the exact category of finite-rank vector bundles on $\text{Gr}^{k,n}_{S}$, and $D^{b}\text{Vect}(\Gr^{k,n}_{S})$ refers to its bounded derived category. See \cite[\S 3.1.3]{mySedano} for details.

\begin{proposition}[\cite{kapranov1988derived}] \label{prop:kapranovexceptional}
    The triangulated category $ D^{b}\text{Vect}(\Gr^{k,n}_{\mathbb{F}})$ admits a full exceptional collection indexed by the totally ordered set $(Y(k,n-k),\prec)$:
\[
 D^{b}\text{Vect}(\Gr^{k,n}_{\mathbb{F}}) = \left\langle S^{\lambda}\mathcal{Q}|\lambda \in Y(k,n-k)\right\rangle.
\]
The total order $\prec$ on $Y(k,n-k)$ has the property that whenever we have a strict inequality $|\alpha| < |\beta|$, then $\alpha \prec \beta$.
\end{proposition}

\begin{remark} \label{remark:SODbasechange}
    By \cite[Proposition 5.1]{kuznetsov_base_2011}, we have an induced semi-orthogonal decomposition on $D^{b}\text{Vect}(\Gr^{k,n}_{S})$ given by $ D^{b}\text{Vect}(\Gr^{k,n}_{S}) = \left\langle S^{\lambda}\mathcal{Q} \otimes p^{*}D^{b}\text{Vect}(S)|\lambda \in Y(k,n-k)\right\rangle$. We use this semi-orthogonal decomposition to compute the Hermitian K-theory of Grassmannians. 
\end{remark}

\begin{remark}\label{NoInclusionImpliesOrthogonality}
    When two Young diagrams $\alpha$ and $\beta$ within $Y(k,n-k)$ don't have an inclusion relation, $S^{\alpha}\mathcal{Q}$ and $S^{\beta}\mathcal{Q}$ are orthogonal to each other. Thus any linear order on $Y(k,n-k)$ refining the partial order given by inclusion can define a full exceptional collection in the above proposition.
\end{remark}
\begin{remark}
  Throughout this paper, we will also be working with $\text{sPerf}(\text{Gr}^{k,n}_{S})$, the dg-category of strictly perfect complexes on $\text{Gr}^{k,n}_{S}$. That is, the dg category of bounded complexes in $\text{Vect}(\text{Gr}^{k,n}_{S})$. 
  Moreover, if we let $w$ be the class of quasi-isomorphisms in $\text{sPerf}(\text{Gr}^{k,n}_{S})$, then the localization $w^{-1}\text{sPerf}(\Gr^{k,n}_{S})$ is the triangulated category $D^{b}\text{Vect}(\Gr^{k,n}_{S})$ (see \cite[Example 1.27]{schlichting17}).
\end{remark}
In practice, the following rule is quite important in computation.
\begin{proposition}[Littlewood--Richardson Rule]\label{prop:LRrule}
For two Young diagrams $\lambda,\mu$, we have
\begin{equation*}
    S^\lambda(\mathcal{Q}) \otimes S^{\mu}(\mathcal{Q}) = \bigoplus\limits_{\nu}  c^\nu_{\lambda,\mu} S^{\nu}Q,
\end{equation*}
where $c^{\nu}_{\lambda,\mu}$ counts Littlewood--Richardson tableaux of skew shape $\nu/\lambda$ with content $\mu$.  
\end{proposition}
We refer the reader to Section \ref{appendix:combinatorics} for a reminder of the relevant combinatorial definitions.
\begin{remark}[\cite{knutson2001honeycombs}]
    Given $\mu,\lambda,\nu\in Y(k,n-k)$, for $c^{\nu}_{\lambda,\mu}$ to be non-zero, $\nu$ should satisfy $\nu_{i+j - 1} \leq \lambda_{i} + \mu_{j}$ for all $1\leq i,j, i + j - 1\leq k$.
\end{remark}

\begin{remark}\label{PieriIdentity}
    One special case of Littlewood--Richardson rule is the \textit{Pieri identity}:
\begin{align} \label{iso:schurtensor}
S^{\lambda}(\mathcal{Q}) \otimes \Lambda^{m}{\mathcal{Q}} \cong \bigoplus_{\lambda(m)} S^{\lambda(m)}(\mathcal{Q}),
\end{align}
where the sum is over all Young diagrams $\lambda(m)$ obtained from $\lambda$ by adding $m$ boxes, with no two in the same row. In particular, for $\lambda = (\lambda_{1},\dots,\lambda_{k})$ we have
\[
S^{\lambda}(\mathcal{Q}) \otimes \det{\mathcal{Q}} \cong  S^{(\lambda_{1}+1,\dots,\lambda_{k}+1)}(\mathcal{Q}).
\]
This allows us to define
\[
S^{\lambda}(\mathcal{Q}) \coloneqq S^{\mu}(\mathcal{Q}) \otimes (\det\mathcal{Q})^{\otimes c}
\]
\textit{for any} $\lambda = \lambda_{1} \geq \cdots \geq \lambda_{k}$, where $\mu = \mu_{1}\geq \dots \geq \mu_{k} \geq 0$ and $c \in \mathbb{Z}$ are such that $\lambda_{i} = \mu_{i} + c$ for all $i$.
With this notion, we have 
\begin{equation}
    S^\lambda(\mathcal{Q})^\vee \cong S^{\lambda^c}(\mathcal{Q}) \otimes (\det\mathcal{Q})^{\otimes -(n - k)} = S^{(-\lambda_{k},\dots,-\lambda_{1})}(\mathcal{Q}).
\end{equation}
which follows from Lemma \ref{lemma:nondegenerate pairing} below.

\end{remark}


\subsection{A canonical pairing}\label{appendix:pairing}
In this subsection, we prove a criterion characterizing the (skew-)symmetry of the canonical pairing attached to a symmetric Young diagram $\lambda$, which will be used in the proof of \Cref{thm:compute}. For the readers' convenience, this subsection may be skipped on a first reading.

Recall first that for $0\leq i \leq k$, there is a natural non-degenerate pairing, given by wedging to the top degree:
\[
\wedge: \wedge^i\mathcal{Q} \otimes \wedge^{k - i}\mathcal{Q} \rightarrow \det(\mathcal{Q}).
\]
For $\lambda\in Y(k,n-k)$, we define $B_\lambda(\mathcal{Q})$ to be the tensor product of these pairing:
\begin{equation*}
   B_\lambda(\mathcal{Q}): \wedge^{\lambda^t}\mathcal{Q} \otimes \wedge^{(\lambda^c)^t}\mathcal{Q}\xrightarrow[]{\wedge} T^{n-k}\det(\mathcal{Q}) = \det(\mathcal{Q})^{\otimes (n - k)},
\end{equation*}
which is again non-degenerate.

\begin{lemma}\label{lemma:nondegenerate pairing}
Given a Young diagram $\lambda\in Y(k,n-k)$, there is a non-degenerate pairing:
\begin{equation*}
    S^\lambda \mathcal{Q} \otimes S^{\lambda^c} \mathcal{Q} \rightarrow \det(\mathcal{Q})^{\otimes (n - k)},
\end{equation*}
which is induced from $B_\lambda(\mathcal{Q})$:
\begin{equation*}
   B_\lambda(\mathcal{Q}): \wedge^{\lambda^t}\mathcal{Q} \otimes \wedge^{(\lambda^c)^t}\mathcal{Q}\xrightarrow[]{\wedge} T^{n-k}\det(\mathcal{Q}) = \det(\mathcal{Q})^{\otimes (n - k)}.
\end{equation*}
\end{lemma}
\begin{proof}
    This is a restatement of \cite[Proposition II.4.2]{AKIN1982207}. Note that in characteristic $0$, Weyl functors (also called coSchur functors) are isomorphic to Schur functors.
\end{proof}

\begin{corollary}
    Given a symmetric Young diagram $\lambda$, there is a canonical non-degenerate pairing:
    \begin{equation*}
        S^{\lambda} \mathcal{Q} \otimes S^{\lambda} \mathcal{Q} \rightarrow \det(\mathcal{Q})^{\otimes (n - k)}
    \end{equation*}
    which is induced from $B_\lambda(\mathcal{Q})$:
\begin{equation*}
    B_\lambda(\mathcal{Q}):\wedge^{\lambda^t}\mathcal{Q} \otimes \wedge^{\lambda^t}\mathcal{Q}\xrightarrow[]{\wedge} \det(\mathcal{Q})^{\otimes (n - k)}.
\end{equation*}
\end{corollary}

\begin{definition}
    A symmetric Young diagram $\lambda\in Y(k,n-k)$ is called \textit{even} if $\sum\limits_{i = 1}^{ n- k} \lambda^t_{i}\lambda_{n - k + 1 - i}^t$ is even, and it is called \textit{odd} if $\sum\limits_{i = 1}^{ n- k} \lambda^t_{i}\lambda_{n - k + 1 - i}^t$ is odd.
\end{definition}

\begin{remark}
    Note that all terms in the sum $\sum\limits_{i = 1}^{ n- k} \lambda^t_{i}\lambda_{n - k + 1 - i}^t$ appear twice except the term $\left(\lambda^t_{\tfrac{n-k+1}{2}}\right)^2$ when $n - k$ is odd. Hence when $n - k$ is even, all diagrams are even; and when $n - k$ is odd, the parity of $\lambda$ is the parity of $\lambda^t_{\tfrac{n-k+1}{2}}$.
\end{remark}

\begin{lemma}
    Given a symmetric Young diagram $\lambda\in Y(k,n-k)$, the canonical pairing $S^{\lambda} \mathcal{Q} \otimes S^{\lambda} \mathcal{Q} \rightarrow  \det(\mathcal{Q})^{\otimes (n-k)}$ is symmetric when $\lambda$ is even, and it is skew-symmetric when $\lambda$ is odd.
\end{lemma}

\begin{proof}
    The pairing $\wedge^{\lambda^t}\mathcal{Q} \otimes \wedge^{\lambda^t} \mathcal{Q} \rightarrow \det(\mathcal{Q})^{\otimes (n - k)}$ is symmetric when $\sum\limits_{i = 1}^{ n- k} \lambda^t_{i}\lambda_{n - k + 1 - i}^t$ is even, and skew-symmetric when $\sum\limits_{i = 1}^{ n- k} \lambda^t_{i}\lambda_{n - k + 1 - i}^t$ is odd. Indeed, this follows because:
\begin{equation*}
    (x_1\wedge\cdots\wedge x_{i}) \wedge (y_1\wedge \cdots \wedge y_j) = (-1)^{ij}(y_1\wedge \cdots \wedge y_j) \wedge (x_1\wedge\cdots\wedge x_{i}). \qedhere
\end{equation*}
\end{proof}

\begin{proposition}\label{canonicalpairing}
    Assume $k(n- k)$ is even. Then for any symmetric $\lambda\in Y(k,n-k)$, the canonical pairing is skew-symmetric when $n - k$ is odd and $k\equiv 2 (\text{mod }4)$, and it is symmetric in all other cases.
\end{proposition}

\begin{proof}
If $n - k$ is odd, then $k$ is even, and
\begin{equation*}
    \sum\limits_{i = 1}^{ n- k} \lambda^t_{i}\lambda_{n - k + 1 - i}^t = 2 \left(\sum\limits_{i = 1}^{\frac{n - k - 1}{2}}\lambda^t_{i}\lambda_{n - k + 1 - i}^t\right) + \left(\lambda^t_{\frac{n - k + 1}{2}}\right)^{2} \equiv \lambda^t_{\frac{n - k + 1}{2}} (\text{mod 2}).
\end{equation*}    
Since $\lambda \in Y(k,n-k)$ is symmetric, we have
\begin{equation*}
    \lambda_i + \lambda_{k + 1 - i} = n - k,
\end{equation*}
and for $1\leq i \leq \lfloor\frac{k+ 1}{2}\rfloor$, $\lambda_i \geq \lambda_{k + 1 - i}$, hence
$\lambda_i \geq \lfloor\frac{n - k + 1}{2}\rfloor$. Therefore, $\lambda^t_{\frac{n - k + 1}{2}}$, which is equal to the number of $i$ such that $\lambda_i\geq\frac{n - k + 1}{2}$, is exactly $\lfloor\frac{k+ 1}{2}\rfloor = \frac{k}{2}$. Hence, when $k \equiv 0 (\text{mod }4)$, the canonical map is symmetric, and when $k \equiv 2 (\text{mod } 4)$, the canonical map is skew-symmetric.

If $n - k$ is even, then
\begin{equation*}
    \sum\limits_{i = 1}^{ n- k} \lambda^t_{i}\lambda_{n - k + 1 - i}^t = 2 \sum\limits_{i = 1}^{\frac{n - k}{2}}\lambda^t_{i}\lambda_{n - k + 1 - i}^t,
\end{equation*}
and the canonical map is always symmetric.
\end{proof}

\subsection{Computation of a combinatorial number}\label{appendix:combinatorics}

In this subsection, we compute in \Cref{finalsteplemma} a combinatorial number which will be used in \Cref{ThmMutation}. For the readers' convenience, this subsection may be skipped on a first reading.

We first recall some definitions from combinatorics.

Recall first that a \textit{tableau} is a Young diagram with an entry in $\mathbb{N}$ placed in each box. A skew tableau is a skew Young diagram with an entry placed in each box. Plainly speaking, there might be missing boxes in its upper-left corner in a skew diagram, as demonstrated in Figure $2$.

\begin{definition}
    Given a (skew) tableau $T$, the \textbf{reverse word} of $T$ is given by reading the entries of $T$ from right to left and top to bottom. 
\end{definition}

\begin{definition}
    A word $w = x_1,\cdots, x_r$ is called a \textbf{lattice word} if, when it is read from the beginning to any letter, the sequence $x_1,\cdots,x_k$ contains at least as many $1$'s as it does $2$'s, at least as many $2$'s as $3$'s, and so on for all positive integers.
\end{definition}

\begin{definition}
    A (skew) tableau $T$ is called \textbf{semistandard} if the entries of $T$ weakly increase along each row and strictly increase down each column.
\end{definition}

\begin{definition}
    A (skew) tableau $T$ is a \textbf{Littlewood--Richardson (skew) tableau} if it is a semi-standard tableau and its reverse word is a lattice word.
\end{definition}

\begin{definition}
    A semi-standard (skew) tableau $T$ is said to have \textbf{content} $\mu = (\mu_1, \dots , \mu_l)$ if its entries consist of $\mu_1$ 1's, $\mu_2$ 2's, and so on  up to $\mu_l$ $l$'s. 
\end{definition}

\begin{remark}
    For Young diagrams $\mu,\lambda,\gamma$, the Littlewood--Richardson coefficient $c^\gamma_{\mu,\lambda}$ counts the number of Littlewood--Richardson tableaux of shape $\gamma/\lambda$ with content $\mu$. It is symmetric with respect to $\lambda$ and $\mu$, though this is not obvious.
\end{remark}

\begin{lemma}\label[lemma]{finalsteplemma}
    For any $\lambda\in Y(k -1, n - k)$, we have
    \begin{equation*}
        c^{(n - k)^{(k - 1)}}_{\lambda,\overline{-(n - k,\lambda)}} = 1.
    \end{equation*}
\end{lemma}
\begin{proof}
    Firstly, define $\lambda^\circ$ to be the skew Young diagram obtained from $\lambda$ by rotating $180$ degrees. Note that the skew diagram $(n - k)^{k - 1}/ \overline{-(n - k,\lambda)}$ is just $\lambda^\circ$ after some possible shift to the left. It suffices to prove that there is only one Littlewood--Richardson skew tableau of shape $\lambda^\circ$ with content $\lambda$. Indeed, since $\lambda^\circ$ has exactly $\lambda_1$-columns, and there can be at most one $1$ in each column due to semi-standardness, the filling of $1$'s is fixed, namely in each column's first box. Now there are only $\lambda_2$ columns with empty boxes. Using the same argument, we know that we can only put $2$ in the first empty box of each column left. Repeat this argument for $3,\cdots,k - 1$, we see that there is only one way to fill in these boxes such that the tableau is semi-standard. The statement that the reverse word of this tableau is a lattice word is clear from the way we fill it in. 
\end{proof}

\begin{figure}[h]
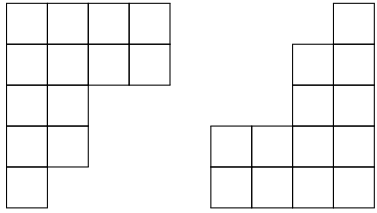
 
\ytableausetup{notabloids}
\begin{ytableau}
\none &  &  &  &     & \none & \none & \none & \none &  \\
\none & & &       &       & \none & \none & \none &  &   \\
\none &  & &  \none    &   \none    & \none&\none  & \none &  & \\
\none &  &  &    \none   & \none& \none &  &  & &\\
\none &   &    \none   &   \none    & \none & \none & & & & 
\end{ytableau}
\caption{The diagram on the left is $\lambda$, the one on the right is $\lambda^\circ$}

\par\smallskip
{\small\noindent \textbf{Alt text:} A Young diagram with row lengths four, four, two, two, and one is shown on the left. Its 180-degree rotation, denoted lambda circle, is shown on the right. \par}
\end{figure}

\section{$\mathbb{G}W$-spectrum of Grassmannians}\label{sec:computation}
In this section, we compute the $\mathbb{G}W$-spectrum of Grassmannians. We will be using the framework for Hermitian K-theory as developed in \cite{schlichting17}. Specifically, for all $m \in \mathbb{Z}$, we compute the Karoubi-Grothendieck-Witt  spectrum $\mathbb{G}W^{[r]}(\text{Gr}^{k,n}_{S}, \det(\mathcal{Q})^{\otimes m} )$ in terms of the $\mathbb{K}$-theory spectrum and $\mathbb{G}W^{[r]}$-spectrum of $S$. This is the content of Theorem \ref{thm:compute} and Theorem \ref{thm:split}. 
Here, $\mathbb{G}W^{[r]}(\text{Gr}^{k,n}_{S}, \det(\mathcal{Q})^{\otimes m} )$ means applying the $\mathbb{G}W^{[r]}$ functor to the dg-category $\text{sPerf}(\text{Gr}^{k,n}_{S})$, equipped with its set of quasi-isomorphisms and duality given by the line bundle $\det(\mathcal{Q})^{\otimes m}$. See \cite[Section 9]{schlichting17} for specific details about this framework in the context of divisorial schemes. Note that we have two-periodicity in powers of the line bundle and four-periodicity in the shift of the upper index.

The following lemma, a generalisation of the additivity theorem, will be crucial in computing the $\mathbb{G}W$-spectrum of Grassmannians. It was proven in the appendix of  \cite{karoubi2021grothendieck}. 

In the following, if $(\mathcal{A}, w, \vee)$ is a dg category with weak equivalences and duality, we write $\mathcal{T}\mathcal{A}\coloneqq w^{-1}\mathcal{A}$, the triangulated category with duality obtained from $\mathcal{A}$ by formally inverting the weak equivalences. Also, the associated hyperbolic category is $\mathcal{HA}\coloneqq \mathcal{A} \times \mathcal{A}^{op}$ and $\mathbb{G}W(\mathcal{H}\mathcal{A},w\times w^{op})  \cong \mathbb{K}(\mathcal{A},w)$. (See \cite{schlichting17} for more details).

\begin{lemma}[Generalised additivity theorem {\cite[Lemma A.2]{karoubi2021grothendieck}}] \label{lemma:gadditivity}
Let $(\mathcal{U},w,\vee)$ be a pretriangulated dg category with weak equivalence and duality such that $\frac{1}{2} \in \mathcal{U}$. Let $\mathcal{A}$ and $\mathcal{B}$ be full pretriangulated dg subcategories of $\mathcal{U}$ containing the $w$-acyclic objects of $\mathcal{U}$. Assume that 
\begin{itemize}
    \item $\mathcal{B}^{\vee} = \mathcal{B}$;
    \item $\mathcal{T}\mathcal{U}(X,Y) = 0$ for all $(X,Y)$ in $\mathcal{A}^{\vee}\times \mathcal{B}, \mathcal{B}\times \mathcal{A}$ or $\mathcal{A}^{\vee} \times \mathcal{A}$; and 
    \item $\mathcal{T}\mathcal{U}$ is generated as a triangulated category by $\mathcal{T}\mathcal{A}$, $\mathcal{T}\mathcal{B}$, $\mathcal{T}\mathcal{A}^{\vee}$.
\end{itemize}
Then, the exact dg form functor 
\[
\mathcal{B}\times \mathcal{H}\mathcal{A} \rightarrow \mathcal{U}: X,(Y,Z) \mapsto X \oplus Y \oplus Z^{\vee}
\]
induces a stable equivalence of Karoubi-Grothendieck-Witt spectra:
\[
\mathbb{G}W(\mathcal{B},w) \times \mathbb{K}(\mathcal{A},w) = \mathbb{G}W(\mathcal{B},w) \times \mathbb{G}W(\mathcal{H}\mathcal{A},w\times w^{op}) \xrightarrow{\sim} \mathbb{G}W(\mathcal{U},w).
\]
\end{lemma} 


This generalised additivity theorem allows us to make the following computation:
\begin{theorem}\label{thm:compute}
   In the above notation, we have 
   \[
   \mathbb{G}W^{[r]}(\Gr^{k,n}_{S},(\det\mathcal{Q})^{\otimes (n-k)}) \cong 
\begin{cases}
\mathbb{K}(S)^{\oplus \frac{1}{2}R(k,n-k)} & k(n-k) \,\, \text{odd,} \\
\mathbb{G}W^{[r - 2]}(S)^{\oplus S(k,n-k)} \oplus \mathbb{K}(S)^{\oplus \frac{1}{2}A(k,n-k)} & (n-k) \,\,  \text{odd and } k\equiv 2 (\text{mod }4),\\
\mathbb{G}W^{[r]}(S)^{\oplus S(k,n-k)} \oplus \mathbb{K}(S)^{\oplus \frac{1}{2}A(k,n-k)} &  \text{otherwise.} \\
\end{cases}
   \]
\end{theorem}
\begin{proof}

By Remark \ref{PieriIdentity}, we have
\begin{align*}
    \mathcal{H}om\left( S^{\lambda}(\mathcal{Q}), (\det\mathcal{Q})^{\otimes (n - k)}\right) &\cong \mathcal{H}om(S^{\lambda}(\mathcal{Q}),\mathcal{O})\otimes (\det\mathcal{Q})^{\otimes (n - k)} \\
    &\cong S^{(-\lambda_{k},\dots,-\lambda_{1})}(\mathcal{Q})\otimes (\det\mathcal{Q})^{\otimes (n - k)} \\
   &\cong S^{(n-k-\lambda_{k},\dots,n-k-\lambda_{1})}(\mathcal{Q}) \\
    & =  S^{\lambda^{c}}(\mathcal{Q}).
\end{align*}
Thus, under this duality, a Young diagram is mapped to its complement within the rectangle.

Define $\mathcal{U}$ as the full dg-subcategory of $\text{sPerf}(\Gr^{k,n}_S)$ on the strictly perfect complexes lying in the full triangulated subcategory of $D^{b}\text{Vect}(\Gr^{k,n}_{S})$ generated, as a triangulated subcategory, by the semi-orthogonal decomposition in Remark \ref{remark:SODbasechange}. We deduce $\mathcal{U}$ is closed under this duality, as for any $E\in D^{b}\text{Vect}(S)$,
\[
\mathcal{H}om\left( S^{\lambda}\mathcal{Q} \otimes p^{*}E, (\det\mathcal{Q})^{\otimes (n-k)} \right) \cong S^{\lambda^{c}}\mathcal{Q} \otimes \mathcal{H}om(p^{*}(E),\mathcal{O}) \in S^{\lambda^{c}}\mathcal{Q} \otimes p^* D^{b}\text{Vect}(S).
\]


In the following, it will be useful to introduce some notation. 
When $kl$ is even, let
\[  
Y_{\text{mid}} = \{ \lambda \in Y(k,l): |\lambda| = kl/2 \},
\qquad
Y_{\text{sym}} = \{\lambda \in Y(k,l) : \lambda^c = \lambda \}.
\]
Choose a subset $Y_0\subset Y_{\text{mid}} \backslash Y_{\text{sym}}$ containing exactly one element from every two-element pair $\{ \lambda,\lambda^c\}$. Define
\[
\mathcal{I}(k,l) = \{\lambda\in Y(k,l):|\lambda| < kl/2\}\cup Y_0,
\qquad
\mathcal{S}ym(k,l) = Y_{\text{sym}},
\qquad
\mathcal{I}^{c}(k,l) = \{ \lambda^c\in Y(k,l): \lambda\in \mathcal{I}(k,l)\}.
\]
When $kl$ is odd, take $\mathcal{S}ym(k,l) = \varnothing$. 
Then, these three sets partition $Y(k,l)$.
Note that when two different Young diagrams $\alpha,\beta$ have the same cardinality, neither diagram contains the other. Hence $S^\alpha \mathcal{Q}$ and $S^\beta \mathcal{Q}$ are orthogonal to each other by Remark \ref{NoInclusionImpliesOrthogonality}. We may therefore choose the linear order on $Y(k,l)$ so that 
\[
\mathcal{I}(k,l) \text{ precedes } \mathcal{S}ym(k,l) \text{ precedes } \mathcal{I}^{c}(k,l).
\]

Moreover, for a set of Young diagrams $J$, we define
\begin{equation*}
    \mathcal{D}_J \coloneqq \left< S^\lambda \mathcal{Q}\otimes p^* D^{b}\text{Vect}(S) \mid \lambda \in J \right>
\end{equation*}
as the full triangulated subcategory generated by these components.

Given these notations, we define $\mathcal{A}$ as the full dg-subcategory of $\mathcal{U}$ on the strictly perfect complexes lying in the full triangulated subcategory $\mathcal{D}_{\mathcal{I}(k,n-k)}$. We define $\mathcal{B}$ as the full dg-subcategory of $\mathcal{U}$ on the strictly perfect complexes lying in the full triangulated subcategory $\mathcal{D}_{\mathcal{S}ym(k,n-k)}$. Let $w$ be the set of quasi-isomorphisms in $\text{sPerf($\Gr^{k,n}_{S}$)}$. Then, by construction, $\mathcal{A}$ and $\mathcal{B}$ satisfy the hypotheses of Lemma \ref{lemma:gadditivity}.
Therefore, we compute 
\[
\mathbb{G}W^{[r]}(\Gr^{k,n}_{S},(\det\mathcal{Q})^{\otimes (n-k)}) \cong \mathbb{K}(\mathcal{A})\oplus \mathbb{G}W^{[r]}(\mathcal{B}),
\]
the duality on $\mathcal{B}$ being the restriction of the original duality.

We want to use additivity of $\mathbb{K}$-theory \cite[Theorem 3.2.27]{mySedano} to show that $\mathbb{K}(\mathcal{A}) \cong \mathbb{K}(S)^{\oplus \frac{1}{2}A(k,n-k)}$. For $\lambda\in\mathcal{I}(k,n - k)$, define $\mathcal{A}_{\lambda}$ as the full dg-subcategory of $\mathcal{A}$ generated on the strictly perfect complexes lying in $S^{\lambda}(\mathcal{Q})\otimes p^{*}D^{b}\text{Vect}(S)$. Then, we want to show that this equivalence is induced by dg-functors
\begin{align*}
   \text{sPerf}(S)  &\rightarrow  \mathcal{A}_{\lambda} \\
    \mathcal{F} \mapsto & S^{\lambda}(\mathcal{Q}) \otimes p^{*}(\mathcal{F}).
\end{align*}
To see this, it suffices to prove that the unit of adjunction $\text{Id} \Rightarrow Rp_{*}\left( (\mathcal{H}om(S^{\lambda}\mathcal{Q}, \mathcal{O})\otimes -)\circ (S^{\lambda}\mathcal{Q}\otimes Lp^{*}(-))\right)$ is a natural isomorphism on derived categories. For this, it suffices to prove that $\mathcal{O}_{\mathbb{F}} \rightarrow Rp_{*}\left( \mathcal{RH}om(S^{\lambda}\mathcal{Q}, S^{\lambda}\mathcal{Q})\right)$ is a quasi-isomorphism. This follows from the fact that $H^{i}\left(\text{Gr}^{k,n}_{\mathbb{F}}, \mathcal{H}om(S^{\lambda}\mathcal{Q}, \mathcal{O}) \otimes S^{\lambda}\mathcal{Q} \right) = 0$ for all $i > 0$, and $R\text{Hom}(S^{\lambda}\mathcal{Q},S^{\lambda}\mathcal{Q} ) = \mathbb{F}$. (The former statement is proven in Lemma \ref{lemma:0end} using Theorem \ref{thm:BWB}.)

For $\lambda = \lambda^{c}$, define $\mathcal{B}_{\lambda}$ as the full dg-subcategory of $\mathcal{B}$ generated on the strictly perfect complexes lying in $S^{\lambda}(\mathcal{Q})\otimes p^{*}D^{b}\text{Vect}(S)$. Then, by additivity for $\mathbb{G}W^{[r]}$ \cite[Theorem 8.10]{schlichting17}, we have 
\[
\mathbb{G}W^{[r]}(\mathcal{B}) \cong \bigoplus_{\lambda = \lambda^{c}}\mathbb{G}W^{[r]}\left(\left(  \mathcal{B}_{\lambda}, w,\vee, can \right) \right),
\]
where the dualities are again given by restriction. 

Note that we obtain a dg form functor (cf. \cite[Remark 1.32]{schlichting17} and \cite[Proposition 7.1.6]{rohrbach2021atiyah})  
\begin{align*}
    \left(\text{sPerf}(S), quis, \mathcal{H}om(-,\mathcal{O}), \epsilon can \right) &\rightarrow \left(  \mathcal{B}_{\lambda}, w,\vee, can \right) \\
    \mathcal{F} \mapsto  S^{\lambda}(\mathcal{Q}) \otimes p^{*}(\mathcal{F})
\end{align*}
only when $\epsilon = \pm 1$ \textit{has the correct sign}. Here, $quis$ denotes the set of quasi-isomorphisms and $can$ denotes the canonical double dual identifications. Given the correct sign, this then induces an equivalence on the corresponding $\mathbb{G}W^{[r]}$-spectra. 

When the canonical pairing $S^{\lambda}(\mathcal{Q}) \otimes S^{\lambda}(\mathcal{Q}) \rightarrow \det(\mathcal{Q})^{\otimes (n-k)}$ is \textit{symmetric}, the form-functor is well-defined for $\epsilon = 1$. When the canonical pairing is \textit{skew-symmetric}, the form-functor is well-defined for $\epsilon = -1$. Combined with Proposition \ref{canonicalpairing}, this gives the statement of the theorem.
\end{proof}

Next, note that $H^0(\Gr^{k,n}_{\mathbb{F}},\mathcal{Q}) = \text{Hom}(\mathbb{F}^n, \mathbb{F})$, hence a non-zero linear form $\psi:\mathbb{F}^n\rightarrow \mathbb{F}$ naturally gives a global section of $\mathcal{Q}$, and the zero locus of this section is the variety of subspaces on which $\psi$ restricts to zero. So we have:
\begin{equation*}
    j:\Gr^{k,n-1}_{\mathbb{F}} \cong \Gr(k,\ker(\psi)) \hookrightarrow \Gr^{k,n}_{\mathbb{F}}.
\end{equation*}
Moreover, $j$ satisfies $j^*\mathcal{Q}_{k,n} = \mathcal{Q}_{k,n-1}$.

In the same way, $H^0(\Gr^{k,n}_{\mathbb{F}},\mathcal{R}^{\vee}) = \mathbb{F}^n$, hence a non-zero vector $u$ in $\mathbb{F}^n$ gives a global section of $\mathcal{R}^\vee$, and the zero locus of this section is the variety of subspaces containing $u$. So we have:
\begin{equation*}
    i:\Gr^{k-1,n-1}_{\mathbb{F}} \cong \Gr(k-1,\mathbb{F}^n/\mathbb{F}u) \hookrightarrow \Gr^{k,n}_{\mathbb{F}}
\end{equation*}
Moreover, $i$ satisfies $i^*\mathcal{Q}_{k,n} = \mathcal{Q}_{k-1,n-1} \oplus \mathcal{O}_{\Gr^{k-1,n-1}_{\mathbb{F}}}$ and $i^*\mathcal{R}_{k,n} = \mathcal{R}_{k-1,n-1}$.

In the following, we fix $\psi$ and $u$ such that $\psi(u)\neq 0$, hence $i(\Gr^{k-1,n-1}_{\mathbb{F}})\cap j(\Gr^{k,n-1}_{\mathbb{F}}) = \emptyset$. These properties are preserved by base change, so that we have maps $i:\Gr_{S}^{k-1,n-1} \rightarrow \Gr_{S}^{k,n}$ and $j: \Gr_{S}^{k,n-1} \rightarrow \Gr_{S}^{k,n}$ such that $i(\Gr^{k-1,n-1}_{S})\cap j(\Gr^{k,n-1}_{S}) = \emptyset$ and the pullback formulae analogous to the above hold.

The maps $i:\Gr_{S}^{k-1,n-1} \rightarrow \Gr_{S}^{k,n}$ and $j: \Gr_{S}^{k,n-1} \rightarrow \Gr_{S}^{k,n}$ will be used to describe the split fibration in Theorem \ref{thm:split}. This split fibration will be a particular instance of the following elementary lemma of triangulated categories:
\begin{lemma} \label{lemma:reverse}
In a triangulated category, suppose we have a commutative diagram 
\[
\begin{tikzcd} 
&Z \arrow[d, "i"] \arrow[dr,"1"] \\
X \arrow[dr,"1",swap] \arrow [r,"f"] & Y \arrow[d, "j"] \arrow[r, "g"] & Z \\
& X
\end{tikzcd},
\]
where the horizontal row is a distinguished triangle and $ji=0$. Then, the sequence 
\[
Z \xrightarrow{i} Y \xrightarrow{j} X
\]
is a distinguished triangle isomorphic to the distinguished triangle $Z \rightarrow X\oplus Z \rightarrow X$.
\end{lemma}

\begin{proof}
We have maps 
\[
\begin{pmatrix}
j \\
g
\end{pmatrix}: Y \rightarrow X\oplus Z, \quad 
\begin{pmatrix}
f & i
\end{pmatrix}: X\oplus Z \rightarrow Y
\]
whose composition 
\[
\begin{pmatrix}
j \\
g
\end{pmatrix}
\begin{pmatrix}
f & i
\end{pmatrix} =
\begin{pmatrix}
jf & ji \\
gf & gi
\end{pmatrix} = 
\begin{pmatrix}
1 & 0 \\
0 & 1
\end{pmatrix}
\]
is the identity. 

Moreover, we have a commutative diagram of distinguished triangles
\[
\begin{tikzcd} 
X \arrow[d,"=" ] \arrow[r, "f"]& Y \arrow[d, "\begin{pmatrix}
j \\ g \end{pmatrix}"] \arrow[r,"g"] & Z \arrow[d, "="] \\
X \arrow[r]  & X\oplus Z \arrow[r] & Z
\end{tikzcd},
\]
so that the middle arrow is an isomorphism. Therefore, $\begin{pmatrix}
f & i\end{pmatrix}: X\oplus Z \rightarrow Y$ is an isomorphism fitting into the commutative diagram 
\[
\begin{tikzcd}
Z \arrow[d,"=" ] \arrow[r]& X\oplus Z \arrow[d, "
(f \,\, i) "] \arrow[r] & X \arrow[d, "="] \\
Z \arrow[r,"i"]  & Y \arrow[r, "j"]  & X 
\end{tikzcd},
\]
so that the bottom row is a distinguished triangle. 
\end{proof}

\begin{theorem}\label{thm:split}
Let $S$ be a divisorial scheme defined over a field of characteristic zero. Then, we have the following split fibration of $\mathbb{G}W$-spectra for all $k\leq n, r \in \mathbb{Z}$:
\begin{align}    \label{sequence:split}
\mathbb{G}W^{[r-(n-k)]}(\Gr^{k-1,n-1}_{S},(\det\mathcal{Q})^{\otimes (n-k)}) \xrightarrow{i_*}\mathbb{G}W^{[r]}(\Gr^{k,n}_{S},(\det\mathcal{Q})^{\otimes (n-k-1)}) \xrightarrow{j^*} \mathbb{G}W^{[r]}(\Gr^{k,n-1}_{S},(\det\mathcal{Q})^{\otimes (n-k-1)}).
\end{align}
\end{theorem}
\begin{proof}
 Let $\mathcal{T}$ be the full dg-subcategory of sPerf$(\text{Gr}^{k,n}_{S})$ on the strictly perfect complexes lying in the full triangulated subcategory $\mathcal{D}_{Y(k,n-k-1)}$.
Let $T \coloneqq w^{-1}\mathcal{T} \subseteq D^{b}\text{Vect}(\Gr^{k,n}_{S})$ be the corresponding triangulated subcategory. Note, $T$ has semi-orthogonal decomposition given by $T = \langle S^{\lambda}\mathcal{Q} \otimes p^{*}D^{b}\text{Vect}(S)\rangle$ for $\lambda\in Y(k,n-k-1)$.
 
 By Remark \ref{PieriIdentity}, we have
 \begin{align*}
     \mathcal{H}om\left( S^{\lambda}(\mathcal{Q}), (\det\mathcal{Q})^{\otimes (n - k - 1)}\right) 
     &\cong S^{(-\lambda_k,\cdots,-\lambda_1)}(\mathcal{Q}) \otimes (\det \mathcal{Q})^{\otimes (n - k - 1)} \\
     &\cong S^{(n - k - 1 - \lambda_k,\cdots,n-k-1-\lambda_1)}(\mathcal{Q}), 
 \end{align*}
 so that for any $E$ in $D^{b}\text{Vect}(S)$,
 \[
 \mathcal{H}om\left( S^{\lambda}\mathcal{Q} \otimes p^{*}E, (\det\mathcal{Q})^{\otimes (n-k-1)} \right) \cong S^{(n - k - 1 - \lambda_k,\cdots,n-k-1-\lambda_1)}\mathcal{Q} \otimes \mathcal{H}om(p^{*}E, \mathcal{O}) \in \mathcal{T}.
 \]
 Hence $\mathcal{T}$ is closed under the duality $\vee$ with values in $(\det\mathcal{Q})^{\otimes(n-k-1)}$. 

 Therefore, by Lemma \ref{lemma:reverse}, the theorem will be proven if we can prove the existence of the following commutative diagram of  $\mathbb{G}W$-spectra, with diagonal arrows equivalences, vertical arrows composing to zero and horizontal row an exact sequence:

\begin{equation} \label{diagram:diamond}
\begin{tikzcd} 
&\mathbb{G}W^{[r-(n-k)]}(\Gr_{S}^{k-1,n-1},(\det\mathcal{Q})^{\otimes(n-k)}) \arrow[d, "i_{*}"] \arrow[dr,"\simeq"] \\
\mathbb{G}W^{[r]}(\mathcal{T},w,\vee) \arrow[dr,"\simeq",swap] \arrow [r] & \mathbb{G}W^{[r]}(\Gr^{k,n}_{S},(\det\mathcal{Q})^{\otimes(n-k-1)}) \arrow[d, "j^{*}"] \arrow[r] & \mathbb{G}W^{[r]}(sPerf(\Gr_{S}^{k,n}),v,\vee) \\
& \mathbb{G}W^{[r]}(\Gr^{k,n-1}_{S},(\det\mathcal{Q})^{\otimes(n-k-1)}).
\end{tikzcd}
\end{equation}

Here, $v$ is the class of maps in $\text{sPerf}(\Gr^{k,n}_{S})$ which become isomorphisms in $D^{b}\text{Vect}(\Gr^{k,n}_{S})/T \simeq T^{\perp}$. The claim that the horizontal row is exact is a consequence of the definition of the class of maps $v$. Furthermore, it is important to note that the latter equivalence is induced by the \textit{left mutation} $L_{T}: D^{b}\text{Vect}(\Gr^{k,n}_{S}) \rightarrow T^{\perp}$ associated to the semi-orthogonal decomposition $D^{b}\text{Vect}(\Gr^{k,n}_{S}) = \langle T^{\perp}, T \rangle$ (The existence of this semi-orthogonal decomposition follows, for example, from \cite[Proposition 2.6 and Theorem 2.10]{bondal1990representable}).


We need to show the existence of the map
\[
 \mathbb{G}W^{[r-(n-k)]}(\Gr_{S}^{k-1,n-1},(\det\mathcal{Q})^{\otimes(n-k)}) \xrightarrow{i_{*}} \mathbb{G}W^{[r]}(\Gr_{S}^{k,n},(\det\mathcal{Q})^{\otimes(n-k-1)}),
\]

induced by the regular embedding $i:\Gr_{S}^{k-1,n-1} \rightarrow \Gr_{S}^{k,n}$. This is an instance of the \textit{push-forward in Hermitian K-theory}, and is constructed for example in Appendix \ref{appendix:marco}, written by Marco Schlichting. In particular, it is proven that for a codimension $c$ regular embedding of divisorial schemes $i : Z \hookrightarrow X$, we have an induced map on $\mathbb{G}W^{[r]}$-spectra
\[
i_{*}:\mathbb{G}W^{[r]}(Z,i^{*}\mathcal{L} \otimes \det{\mathcal{N}}[-c]) \rightarrow \mathbb{G}W^{[r]}(X,\mathcal{L}).
\]
Here, $\mathcal{L}$ is a line bundle on $X$ and $\mathcal{N}$ is the conormal sheaf of the regular embedding $i : Z \hookrightarrow X$. (See also \cite[Corollary 5.1.8]{calmes2024motivic} for the corresponding explanation in the context of Poincar\'{e} $\infty$-categories.)



We compute the determinant of the conormal sheaf in our case. (In this computation of the conormal sheaf, we take $\vee$ to be the usual sheaf dual $\mathcal{H}om(-,\mathcal{O})$ for notational convenience.) 

Note that $\mathcal{N}^{\vee}$ sits inside a short exact sequence 
\[
0 \rightarrow \mathcal{N}^{\vee} \rightarrow i^{*}\Omega_{\Gr_{S}^{k,n}} \rightarrow \Omega_{\Gr_{S}^{k-1,n-1}}\rightarrow 0,
\]
where $\Omega_{\Gr_{S}^{k,n}}$ is the relative cotangent sheaf of $ \Gr_{S}^{k,n} \rightarrow S$. 
Defining $\omega_{\Gr_{S}^{k,n}} \coloneqq \det\Omega_{\Gr_{S}^{k,n}}$ and taking top exterior powers of the above short exact sequence, we have 
\[
i^{*}\omega_{\Gr_{S}^{k,n}} \cong \det \mathcal{N}^{\vee} \otimes \omega_{\Gr_{S}^{k-1,n-1}},
\]
so that 
\[
\det \mathcal{N} \cong \omega_{\Gr_{S}^{k-1,n-1}} \otimes \left(i^{*}\omega_{\Gr_{S}^{k,n}}\right)^{\vee}.
\]
Recall that 
\[
\Omega_{\Gr_{S}^{k,n}} \cong \mathcal{H}om\left(\mathcal{Q},\mathcal{R} \right) \cong \mathcal{Q}^{\vee} \otimes \mathcal{R},
\]
which gives us 
\[
\omega_{\Gr_{S}^{k,n}} \cong \det\left(\mathcal{Q}^{\vee}\otimes\mathcal{R} \right) \cong \left(\det\mathcal{Q}^{\vee}\right)^{\otimes (n-k)}\otimes\left(\det\mathcal{R}\right)^{\otimes k} \cong (\det\mathcal{R})^{\otimes n}.
\]
Therefore, 
\[
    \det\mathcal{N} \cong \left(\det\mathcal{R}\right)^{\otimes (n-1)}\otimes \left(\det\mathcal{R}^{\vee}\right)^{\otimes n}
    \cong \det\mathcal{R}^{\vee}
    \cong \det\mathcal{Q}.
\]

Thus, it follows that the pushforward induces a map 
\[
 \mathbb{G}W^{[r-(n-k)]}(\Gr_{S}^{k-1,n-1},(\det\mathcal{Q})^{\otimes(n-k)}) \xrightarrow{i_{*}} \mathbb{G}W^{[r]}(\Gr_{S}^{k,n},(\det\mathcal{Q})^{\otimes(n-k-1)}).
\]
Moreover, as $j^{*}\det\mathcal{Q}^{\otimes (n-k-1)} \cong \det\mathcal{Q}^{\otimes (n-k-1)}$, it is clear that the bottom vertical map is well-defined (see \cite[\S 9.3]{schlichting17}). 

To see that $j^{*}\circ i_{*} = 0$, recall that the intersection $i(\Gr_{S}^{k-1,n-1}) \cap j( \Gr_{S}^{k,n-1})$ is empty, so that the map $j:\Gr_{S}^{k,n-1} \rightarrow \Gr_{S}^{k,n}$ factors through the \textit{open complement} of the map $i:\Gr_{S}^{k-1,n-1} \rightarrow \Gr_{S}^{k,n}$. The composition $j^{*}\circ i_{*}$ is therefore equal to zero by Lemma \ref{lemma:zerocomp}.



Next, we show that the diagonal arrows are equivalences. We begin with the lower diagonal arrow. 
 
Recall that $T$ has a semi-orthogonal decomposition given by $\left\langle S^{\lambda}\mathcal{Q} \otimes p^{*}D^{b}\text{Vect}(S)\right\rangle_{\lambda \in Y(k,n-k-1)}$. Furthermore, note that for all $\lambda \in Y(k,n-k-1)$,  
\[
j^{*}\left(S^{\lambda}\mathcal{Q} \otimes p^{*}E \right) \simeq j^{*} S^{\lambda}\mathcal{Q} \otimes j^{*} p^{*}E \simeq S^{\lambda}\mathcal{Q} \otimes p^{*}E ,
\] 
where $E \in D^{b}\text{Vect}(S)$. 
Here, we have used that the pullback commutes with the Schur functor, $j^{*}\mathcal{Q} \cong \mathcal{Q}$ and $j^{*}p^{*} = p^{*}$.

Define $\mathcal{A} \subseteq \mathcal{T}$ (resp. $\mathcal{A}' \subseteq \text{sPerf}(\Gr^{k,n-1}_{S})$) to be the full dg subcategory on the strictly perfect complexes lying in the full triangulated subcategory of $T$ (resp. $D^{b}\text{Vect}(\Gr^{k,n-1}_{S})$) corresponding to $\mathcal{D}_{\mathcal{I}(k,n-k-1)}$.

Define $\mathcal{B} \subseteq \mathcal{T}$ (resp. $\mathcal{B}' \subseteq \text{sPerf}(\Gr^{k,n-1}_{S})$) to be the full dg subcategory on the strictly perfect complexes lying in the full triangulated subcategory of $T$ (resp. $D^{b}\text{Vect}(\Gr^{k,n-1}_{S})$) corresponding to $\mathcal{D}_{\mathcal{S}ym(k,n-k-1)}$.

Define filtrations 
\[
0 = \mathcal{F}_{0} \subset \cdots \subset \mathcal{F}_{\frac{1}{2}A(k,n-k-1)} = \mathcal{A},
\]
and 
\[
0 = \mathcal{G}_{0} \subset \cdots \subset \mathcal{G}_{S(k,n-k-1)} = \mathcal{B},
\]
where $\mathcal{F}_{i}$ is the full dg subcategory generated by the first $i$ semi-orthogonal components indexed by $\mathcal{I}(k,n-k-1)$; and $\mathcal{G}_{i}$ is the full dg subcategory generated by the first $i$ semi-orthogonal components indexed by $\mathcal{S}ym(k,n-k-1)$. Define the filtrations 
\[
0 = \mathcal{F'}_{0} \subset \cdots \subset \mathcal{F'}_{\frac{1}{2}A(k,n-k-1)} = \mathcal{A'},
\]
and 
 
\[
0 = \mathcal{G'}_{0} \subset \cdots \subset \mathcal{G'}_{S(k,n-k-1)} = \mathcal{B'},
\]
analogously. 

As the map $j^{*}$ respects these filtrations and the associated quotients $\mathcal{T}\mathcal{F}_{i+1}/\mathcal{T}\mathcal{F}_{i}, \dots, \mathcal{T}\mathcal{G'}_{i+1}/\mathcal{T}\mathcal{G'}_{i}$ are equivalent to the triangulated category generated by a single component of the semi-orthogonal decomposition, it follows by induction that $j^{*}$ induces an equivalence between the $K$-theory of $\mathcal{A}$ and $\mathcal{A'}$; and the Hermitian K-theory of $\mathcal{B}$ and $\mathcal{B'}$.  

By Lemma \ref{lemma:gadditivity} therefore, $j^{*}$ defines an isomorphism between $\mathbb{G}W^{[r]}(\mathcal{T},w,\vee)$ and $\mathbb{G}W^{[r]}(\Gr^{k,n-1}_{S},(\det\mathcal{Q})^{\otimes(n-k-1)})$, the latter being computed in Theorem \ref{thm:compute}. 

Finally, we argue that the upper diagonal arrow is an equivalence.

We point out that $T^{\perp}$ has a semi-orthogonal decomposition given by $\left\langle S^{(\lambda,0)}(\mathcal{Q}) \otimes \det(Q)^{-1} \otimes p^{*}D^{b}\text{Vect}(S)\right\rangle$, and ${}^{\perp}T$ has a semi-orthogonal decomposition given by $\left\langle S^{(n-k,\lambda)}(\mathcal{Q}) \otimes p^{*}D^{b}\text{Vect}(S)\right\rangle$ where $\lambda$ are Young diagrams in $ Y(k-1,n-k)$. To see this, note that the objects in $\langle T^\perp, T\rangle \otimes \det(\mathcal{Q})$ form Kapranov's full exceptional collection, so are objects in $\langle T, {}^{\perp}T\rangle$. 

Furthermore, the equivalence $T^{\perp} \xrightarrow{\simeq} D^{b}\text{Vect}(\text{Gr}_{S}^{k,n})/T$ can be used to pull-back the duality structure on $D^{b}\text{Vect}(\text{Gr}_{S}^{k,n})/T$ to equip $T^{\perp}$ with a  duality such that the equivalence becomes a form functor. See  for example \cite[Lemma 4.3]{balmer2002gersten}. Moreover, we may take the duality on $T^{\perp}$ to be $\# \coloneqq L_{T} \circ \vee$. By \Cref{ThmMutation} and \Cref{remark:mutationbasechange}, we have
\[ 
\left[ S^{(\lambda,0)}(\mathcal{Q}) \otimes \det(Q)^{-1} \otimes p^{*}E \right]^{\#} = S^{(\lambda^{c},0)}(\mathcal{Q}) \otimes \det(Q)^{-1} [n-k] \otimes \mathcal{H}om(p^{*} E, \mathcal{O}),
\]
where $E \in D^{b}\text{Vect}(S)$.
So we again have a notion of asymmetric and symmetric partitions under the duality $\#$. 

Therefore, if we can define triangulated subcategories $\mathcal{V}, \mathcal{W}$ of $T^{\perp}$ such that 
\begin{itemize}
    \item $\mathcal{W}^{\#} = \mathcal{W}$,
    \item $T^{\perp}(X,Y) = 0$ for all $(X,Y)$ in $\mathcal{V}^{\#}\times \mathcal{W}, \mathcal{W}\times \mathcal{V}$ or $\mathcal{V}^{\#} \times \mathcal{V}$,
    \item $T^{\perp} = \langle\mathcal{V},\mathcal{W},\mathcal{V}^{\#}  \rangle$,
\end{itemize}
 we may then use the form functor $T^{\perp} \xrightarrow{\simeq} D^{b}\text{Vect}(\text{Gr}_{S}^{k,n})/T$ to define corresponding triangulated subcategories of $D^{b}\text{Vect}(\text{Gr}_{S}^{k,n})/T$ that behave under duality as above. These triangulated subcategories  then determine the dg-categories needed to use Lemma \ref{lemma:gadditivity}. 

We define $\mathcal{V}$ to be the triangulated subcategory of $T^{\perp}$ generated by the components in the semi-orthogonal decomposition of $T^{\perp}$  corresponding to the partitions in $\mathcal{I}(k-1,n-k)$. We define $\mathcal{W}$ to be the triangulated subcategory of $T^{\perp}$ generated by the components in the semi-orthogonal decomposition of $T^{\perp}$ corresponding to the partitions in $\mathcal{S}ym(k-1,n-k)$. By \Cref{cor:miracle} and \Cref{remark:mutationbasechange}, 
\[
L_{T} \circ i_{*}(S^{\lambda}\mathcal{Q} \otimes p^{*}E) \simeq \left(S^{(\lambda,0)}\mathcal{Q} \otimes \det(\mathcal{Q})^{-1}[n-k]\right) \otimes p^{*}E,
\]
where $E\in D^{b}\text{Vect}(S)$. Hence we may use a similar filtration argument as above to conclude the upper diagonal is an equivalence. This concludes the proof.
\end{proof}

Moreover, as $i(\Gr_{S}^{k-1,n-1}) \cap j(\Gr_{S}^{k,n-1}) = \emptyset$, $j^{*}(S^{\lambda}\mathcal{Q}) = S^{\lambda}\mathcal{Q}$ and $L_{T}\circ i_{*}(S^{\lambda}\mathcal{Q}) \cong S^{(\lambda,0)}\mathcal{Q} \otimes \det(\mathcal{Q})^{-1}[n-k]$, we may use Lemma \ref{lemma:reverse} to deduce the following result:
\begin{corollary} \label[corollary]{cor:Ksplit}
    In the above notation, we have split fibration of $\mathbb{K}$-theory spectra for all $k\leq n, r \in \mathbb{Z}$:
    \[
    \mathbb{K}(\Gr^{k-1,n-1}_{S}) \xrightarrow{i_{*}}\mathbb{K}(\Gr^{k,n}_{S})\xrightarrow{j^{*}} \mathbb{K}(\Gr^{k,n-1}_{S}).
    \] 
\end{corollary}
\begin{remark}
This result is known for $S = \text{Spec}(\mathbb{F})$. Indeed, we have an affine bundle $\Gr^{k,n}_{\mathbb{F}}\setminus \Gr^{k,n-1}_{\mathbb{F}} \rightarrow \Gr^{k-1,n-1}_{\mathbb{F}}$, whose pullback induces an isomorphism $\mathbb{K}(\Gr^{k-1,n-1}_{\mathbb{F}}) \cong \mathbb{K}(\Gr^{k,n}_{\mathbb{F}}\setminus \Gr^{k,n-1}_{\mathbb{F}} )$, giving us a fibration sequence
\[
\mathbb{K}(\Gr^{k,n-1}_{\mathbb{F}}) \xrightarrow{j_{*}}\mathbb{K}(\Gr^{k,n}_{\mathbb{F}})\xrightarrow{i^{*}} \mathbb{K}(\Gr^{k-1,n-1}_{\mathbb{F}}).
\]
One can then apply Lemma \ref{lemma:reverse} to obtain the desired fibration sequence for $S = \text{Spec}(\mathbb{F})$. However, at our level of generality, Corollary \ref{cor:Ksplit} appears to be new.
\end{remark}

Moreover, the $\mathbb{L}$-theory spectrum (\cite[Definition 8.12]{schlichting17}) of Grassmannians may also be computed using the above reasoning. 

To do this, it will be convenient to spell out the version of the Lemma \ref{lemma:gadditivity} for $\mathbb{L}$-theory. (In the notation of Witt groups, this was proven in \cite[Theorem 3.6]{walter2003}). 
\begin{lemma} \label{lemma:ladditivity}
Let $(\mathcal{U},w,\vee)$ be a pretriangulated dg category with weak equivalence and duality such that $\frac{1}{2} \in \mathcal{U}$. Let $\mathcal{A}$ and $\mathcal{B}$ be full pretriangulated dg subcategories of $\mathcal{U}$ containing the $w$-acyclic objects of $\mathcal{U}$. Assume that 
\begin{itemize}
    \item $\mathcal{B}^{\vee} = \mathcal{B}$;
    \item $\mathcal{T}\mathcal{U}(X,Y) = 0$ for all $(X,Y)$ in $\mathcal{A}^{\vee}\times \mathcal{B}, \mathcal{B}\times \mathcal{A}$ or $\mathcal{A}^{\vee} \times \mathcal{A}$; and 
    \item $\mathcal{T}\mathcal{U}$ is generated as a triangulated category by $\mathcal{T}\mathcal{A}$, $\mathcal{T}\mathcal{B}$, $\mathcal{T}\mathcal{A}^{\vee}$.
\end{itemize}
Then, the inclusion 
\[
\mathcal{B} \hookrightarrow \mathcal{U}
\]
induces a stable equivalence of $\mathbb{L}$-spectra:
\[
\mathbb{L}(\mathcal{B},w) \xrightarrow{\sim} \mathbb{L}(\mathcal{U},w).
\]
\end{lemma}

\begin{proof}

Let $\mathcal{A}' \subset \mathcal{U}$ be the full dg subcategory whose objects lie in the triangulated subcategory of $\mathcal{TU}$ generated by $\mathcal{TA}$ and $\mathcal{TA}^{\vee}$. Let $v$ be the class of maps in $\mathcal{U}$ which are isomorphisms in $\mathcal{TU}/\mathcal{TA'}$.

Then, by localisation, the sequence $(\mathcal{A}',w) \rightarrow (\mathcal{U},w) \rightarrow (\mathcal{U},v)$ induces a fibration sequence $\mathbb{L}(\mathcal{A}',w) \rightarrow \mathbb{L}(\mathcal{U},w) \rightarrow \mathbb{L}(\mathcal{U},v)$. By additivity for $\mathbb{G}W$ \cite[Proposition 8.15]{schlichting17}, the map $\mathbb{K}(\mathcal{A}',w)_{hC_{2}} \rightarrow \mathbb{G}W(\mathcal{A}',w)$ is an equivalence, so that by \cite[Theorem 8.13]{schlichting17}, $\mathbb{L}(\mathcal{A}',w) \simeq 0$. The result follows.  
\end{proof}

\begin{corollary}\label[corollary]{cor:Ltheory}
    In the above notation, we have 
    \[
   \mathbb{L}^{[r]}(\Gr^{k,n}_{S},(\det\mathcal{Q})^{\otimes (n-k)}) \cong 
\begin{cases}
0 & k(n-k) \,\, \text{odd,} \\
\mathbb{L}^{[r - 2]}(S)^{\oplus S(k,n-k)} & (n-k) \,\,  \text{odd and } k\equiv 2 (\text{mod }4),\\
\mathbb{L}^{[r]}(S)^{\oplus S(k,n-k)}  &  \text{otherwise } \\
\end{cases}
 \]
 and split fibrations of stabilized $\mathbb{L}$-theory spectra
 \[
\mathbb{L}^{[r-(n-k)]}(\Gr^{k-1,n-1}_{S},(\det\mathcal{Q})^{\otimes (n-k)}) \xrightarrow{i_*}\mathbb{L}^{[r]}(\Gr^{k,n}_{S},(\det\mathcal{Q})^{\otimes (n-k-1)}) \xrightarrow{j^*} \mathbb{L}^{[r]}(\Gr^{k,n-1}_{S},(\det\mathcal{Q})^{\otimes (n-k-1)}).
\]
\end{corollary}
\begin{proof}
The first part of the corollary follows from Lemma \ref{lemma:ladditivity} and the reasoning used in the proof of Theorem \ref{thm:compute}. The proof of second part of the corollary is the same as the proof of Theorem \ref{thm:split}, replacing $\mathbb{G}W^{[r]}$ with $\mathbb{L}^{[r]}$.
\end{proof}

\section{Computation of the mutation}\label{sec:mutation}
In this section we compute two mutations (\Cref{ThmMutation} and \Cref{cor:miracle}) used in the proof of \Cref{thm:split}. The strategy is to firstly prove these statements over field of characteristic zero, then use a base-change argument to obtain the statement for divisorial scheme over a field of characteristic zero (See \Cref{remark:mutationbasechange}). The duality $\vee$ in this section will always be taken to be the usual duality. 


Let $\alpha = (\alpha_1,...,\alpha_n)$ be integers such that $\alpha_1\geq ...\geq\alpha_{k}$, $\alpha_{k+1}\geq...\geq\alpha_n$. Denote by $\beta = (\alpha_1,...,\alpha_k)$, $\gamma = (\alpha_{k+1},...,\alpha_n)$. Define 
\begin{equation*}
    \mathcal{V}(\alpha) \coloneqq S^{\beta}\mathcal{Q} \otimes S^\gamma \mathcal{R}.
\end{equation*}

Set $\rho = (n-1,n-2,...,0)$. A permutation on $n$ letters $\sigma\in S_n$ naturally acts on the set $\mathbb{Z}^n$:
\begin{equation*}
    \sigma((\alpha_1,...,\alpha_n)) = (\alpha_{\sigma(1)},...,\alpha_{\sigma(n)}).
\end{equation*}

We define the dotted action of $S_n$ on $\mathbb{Z}^n$ in the following way:
\begin{equation*}
    (\alpha)\cdot \sigma = \sigma (\alpha) + \sigma(\rho) -\rho.
\end{equation*}

It is direct to verify that the dotted action is a right group action.
\begin{remark}
    From the definition,
    \begin{equation*}
        (\alpha)\cdot \sigma = (\alpha_{\sigma(1)} + 1 - \sigma(1),\alpha_{\sigma(2)}+ 2 - \sigma(2),\cdots,\alpha_{\sigma(n)} + n - \sigma(n)).
    \end{equation*}
    
    For the transposition $(i,j)$ with $i< j $, we have
    \begin{equation*}
        (\alpha_1,\cdots,\alpha_i,\cdots,\alpha_j,\cdots\alpha_n)\cdot (i j ) = (\alpha_1,\cdots,\alpha_j + i - j,\cdots,\alpha_i +j - i,\cdots, \alpha_n),
    \end{equation*}
    where $\alpha_j + i - j$ is placed in the $i^{th}$ place, $\alpha_{i} + j - i$ is placed in the $j^{th}$ place. In particular, when $j = i + 1$, we have
    \begin{equation*}
        (\alpha_1,\cdots,\alpha_i,\alpha_{i + 1},\cdots,\alpha_n)\cdot(i,i+1) = (\alpha_1,\cdots,\alpha_{i+1} - 1,\alpha_i + 1,\cdots,\alpha_n).
    \end{equation*}
    
\end{remark}

The following theorem tells us how to compute the cohomology of $\mathcal{V}(\alpha)$.

\begin{theorem}[Borel--Weil--Bott Theorem, {\cite[Corollary 4.1.9]{Weyman_2003}}]\label{thm:BWB}
    Notations as above, and let $\mathbb{F}$ be a field of characteristic zero. Then one of two mutually exclusive possibilities occurs:
    \begin{enumerate}[(i)]
        \item There exists $\sigma\in S_n$, $\sigma\neq \text{id}$, such that $(\alpha)\cdot\sigma = \alpha$. Then $H^i(\Gr^{k,n}_{\mathbb{F}},\mathcal{V}(\alpha)) = 0$ for all $i\geq 0$.
        \item There exists a unique $\sigma\in S_n$ such that $(\alpha)\cdot\sigma \coloneqq (\nu)$ is a partition (i.e. is non-increasing). In this case all cohomology $H^i(\Gr^{k,n}_{\mathbb{F}},\mathcal{V}(\alpha))$ are zero for $i\neq l(\sigma)$, and 
        \begin{equation*}
            H^{l(\sigma)}(\Gr^{k,n}_{\mathbb{F}},\mathcal{V}(\alpha)) = S^\nu ((\mathbb{F}^{n})^\vee).
        \end{equation*}
    \end{enumerate}
Here, $l(\sigma)$ denotes the minimal number of adjacent transpositions required to write $\sigma$ as a product of adjacent transpositions.
    
\end{theorem}

\begin{remark}
    When using this theorem below, we will suppress the Grassmannian from the notation, as it is clear from context. 
\end{remark}

Recall also that a \textit{Serre functor} for a Hom-finite $\mathbb{F}$-linear triangulated category $\mathcal{C}$ is a $\mathbb{F}$-linear equivalence $\mathcal{S}_{\mathcal{C}} : \mathcal{C}\rightarrow \mathcal{C}$ with functorial isomorphisms
\begin{equation*}
    \Hom_{\mathcal{C}} (E,F ) \simeq \Hom_{\mathcal{C}}(F,\mathcal{S}_{\mathcal{C}}E)^*
\end{equation*}
for all objects $E$ and $F$. Note that, when a Serre functor exists on a triangulated category, it is unique up to unique natural isomorphism; and Serre functors always exist on $D^{b}\text{Vect}(X)$ for any smooth proper variety $X$ over $\mathbb{F}$. 

Moreover, for an admissible subcategory $\mathcal{C}_0 \subset\mathcal{C}$, $\mathcal{S}_{\mathcal{C}}$ and its inverse induce equivalences 

\[
\begin{tikzcd}[column sep=3.5em]
\mathcal{C}_0^{\perp}
\arrow[r, shift left=1.2ex, "S_{\mathcal C}^{-1}"]
&
{}^{\perp}\mathcal{C}_0
\arrow[l, shift left=1.2ex, "\sim"{above}, "S_{\mathcal C}"{below}].
\end{tikzcd}
\]

Serre functors and mutations are closely related, and we will use this relationship to compute our desired mutation.
\begin{lemma}[{\cite[Lemma $7.1.14$]{Huybrechts_2023}}]\label{Lemma on mutations and Serre functors}
    Assume $\mathcal{C}_0 \hookrightarrow \mathcal{C}$ is an admissible subcategory. If $\mathcal{C}$ admits a Serre functor $\mathcal{S}_C$, then also $\mathcal{C}_0$, $\mathcal{C}_0^\perp$, ${}^\perp\mathcal{C}_0$ admit a Serre functor. Moreover,
    \[
    \mathcal{S}^{-1}_{\mathcal{C}_0^{\perp}} \simeq \left(L_{\mathcal{C}_0} \circ \mathcal{S}_{\mathcal{C}}^{-1}\right)|_{\mathcal{C}^\perp_0}.
    \]
\end{lemma}

The following lemmas will also be used in the proof of Theorem \ref{ThmMutation}.

\begin{lemma}[{\cite[Lemma $20.50.5$]{stacks-project}}]\label{LemmaVecBund}
    For $\mathcal{E}, \mathcal{F}$ finite-rank vector bundles on a scheme $X$, we have
    \begin{equation*}
        \mathbf{R}\Hom(\mathcal{E},\mathcal{F}) = \mathbf{R}\Hom(\mathcal{O},\mathcal{F}\otimes \mathcal{E}^{\vee}).
    \end{equation*}
\end{lemma}

\begin{lemma}\label{LemmaSerreFunctor}
    Assume that an $\mathbb{F}$-linear triangulated category $\mathcal{A}$ has a Serre functor $S_{\mathcal{A}}$ and admits a full exceptional collection $\{E_1,\cdots,E_n\}$. Let $j$ be a fixed integer with $1\leq j \leq n$ and let $E'$ be an exceptional object in $\mathcal{A}$. If $E'$ satisfies:
    \begin{enumerate}[(i)]
        \item $\mathbf{R}\textup{Hom}(E_i, E') = 0$ for all $i < j$, 
        \item $\mathbf{R}\textup{Hom}(S_{\mathcal{A}}(E_k),E') = 0$ for all $k>j$,
        \item $\mathbf{R}\textup{Hom}(E_j, E') = \mathbb{F}$,
    \end{enumerate}
    then $E' \cong S_{\mathcal{A}}(E_j)$.
\end{lemma}
\begin{proof}
    Note that $\{S_{\mathcal{A}}(E_1),\cdots,S_{\mathcal{A}}(E_n)\}$ is also a full exceptional collection of $\mathcal{A}$, and the condition says that $E'$ is an exceptional object with $\text{Hom}(S_{\mathcal{A}}(E_k),E') = 0$ for all $k > j$ and $\text{Hom}(E',S_{\mathcal{A}}(E_i)) = 0$ for all $i < j$. Hence, the first two bullet points imply that $E'$ must be isomorphic to direct sum of copies of $S_{\mathcal{A}}(E_j)$ up to a homological shift. The final bullet point implies there is only one copy of $S_{\mathcal{A}}(E_j)$ and no homological shift. 
\end{proof}

For the next two statements, we specialize to $S = \text{Spec } \mathbb{F}$. Recall that we have defined in \Cref{sec:computation} $T = \langle S^{\lambda} \mathcal{Q} | \lambda \in Y(k, n - k - 1) \rangle$, and thus
\begin{equation*}
    {}^{\perp}T = \left\langle S^{(n - k,\lambda)}\mathcal{Q} | \lambda \in Y(k-1,n-k)   \right\rangle,
\end{equation*}
\begin{equation*}
     T^{\perp} = \left\langle S^{(\lambda,0)}\mathcal{Q}\otimes \det (\mathcal{Q})^{-1} | \lambda \in Y(k-1,n-k)  \right\rangle. 
\end{equation*}

\begin{theorem}\label{ThmMutation}
    $L_T(S^{(n-k,\lambda)}\mathcal{Q}) = S^{(\lambda,0)}\mathcal{Q}\otimes \det(\mathcal{Q})^{-1} [n - k]$.
\end{theorem}
\begin{proof}
    By Lemma \ref{Lemma on mutations and Serre functors}, we have
    \begin{equation*}
        L_T|_{^{\perp}T} = \mathcal{S}^{-1}_{T^{\perp}}\circ \mathcal{S}_{\mathcal{D}}|_{^{\perp}T},
    \end{equation*}
    where $\mathcal{S}_{T^{\perp}}$ is the Serre functor of $T^{\perp}$, $\mathcal{S}_{\mathcal{D}}$ is the Serre functor of $\mathcal{D} \coloneqq D^{b}\text{Vect}(\Gr^{k,n}_{\mathbb{F}})$. Therefore, it suffices to prove
    \begin{equation*}
        \mathcal{S}_{\mathcal{D}}(S^{(n-k,\lambda)}\mathcal{Q}[k-n]) = \mathcal{S}_{T^{\perp}}(S^{(\lambda,0)}\mathcal{Q}\otimes \det(\mathcal{Q})^{-1}).
    \end{equation*}

(Note, by functoriality of $\mathcal{S}_{\mathcal{D}}$, we have $\mathcal{S}_{\mathcal{D}}(S^{(n-k,\lambda)}\mathcal{Q})\in T^{\perp}$.)
     
   To use Lemma \ref{LemmaSerreFunctor}, we begin by showing that for a fixed $\lambda\in Y(k-1,n-k)$, 
   \[
   \mathbf{R}\text{Hom}(S^{(\mu,0)}\mathcal{Q}\otimes \det(\mathcal{Q})^{-1},\mathcal{S}_{\mathcal{D}}(S^{(n-k,\lambda)}\mathcal{Q})) = 0
   \]
   for all $\mu \prec  \lambda$. Using the functoriality of Serre functor, this is equivalent to showing
    \begin{equation*}
        \mathbf{R}\text{Hom}(S^{(n - k,\lambda)}\mathcal{Q}, S^{(\mu,0)}\mathcal{Q} \otimes \det(\mathcal{Q})^{-1}) = 0
    \end{equation*}
    for all $\mu \prec  \lambda$. 

Recall that for any sequence of non-increasing integers $\lambda = (\lambda_1,\cdots,\lambda_k)$, we have defined in Section \ref{sec:prelim}
\begin{equation*}
    \overline{\lambda} = (\lambda_1 - \lambda_k,\lambda_2 - \lambda_k,\cdots,\lambda_{k-1}-\lambda_k,0),
\end{equation*}
the `minimal positive' modification of $\lambda$. We will also write $-\lambda = (-\lambda_k, \cdots, -\lambda_1)$. For $r \in \mathbb{Z}, k \in \mathbb{N}$, we will write $(r)^{k} \coloneqq (r,\cdots,r)$, $k$ entries of $r$.
    
    Using this notation, we have
    \begin{align*}
        \text{Hom}(S^{(n - k,\lambda)}\mathcal{Q},S^{(\mu,0)}\mathcal{Q}\otimes \det(\mathcal{Q})^{-1}[i]) &= H^i(S^{(\mu,0)}\mathcal{Q}\otimes S^{\overline{-(n - k,\lambda)}}\mathcal{Q} \otimes \det(\mathcal{R})^{\otimes (n - k + 1)}) \\
        &= H^i\left(\bigoplus_{|\gamma| =|\mu| - |\lambda| + (k - 1)(n - k)}c^{\gamma}_{\mu,\overline{-(n-k,\lambda)}}S^{(\gamma_1,\cdots,\gamma_k)}\mathcal{Q}\otimes \det(\mathcal{R})^{\otimes (n - k + 1)}\right),
    \end{align*}
    and $\gamma_{k - 1}$ satisfies $\gamma_{k - 1} \leq n- k + \mu_i - \lambda_i$ for all $i$. Since $\mu \prec  \lambda$, there is at least one $i$ such that $\mu_i < \lambda_i$. Hence $\gamma_{k - 1} \leq n -k - 1$. We claim that for all such $\gamma$ and all $i \geq 0$,
    \begin{equation*}
        H^i(S^{\gamma}\mathcal{Q}\otimes\det(\mathcal{R})^{\otimes (n - k + 1)}) =0.
    \end{equation*}
Indeed, by Bott's algorithm (see \cite[Remark 4.1.5]{Weyman_2003}), it suffices to permute the element $\gamma_{k-1}$ in the sequence $(\gamma,(n - k+ 1)^{(n - k)})$ to the right by a series of adjacent transpositions via the dotted action until we reach a sequence of the form $(\dots,n-k,n-k+1,\dots)$. Concretely, this is achieved by the permutation $\sigma_{\gamma} = (k-1,k)(k,k+1)\cdots(n-\gamma_{k-1}-2,n-\gamma_{k-1}-1)$.
We may then conclude by the Borel--Weil--Bott Theorem.

      Next, for a fixed $\lambda$, we may assume that we have proved inductively 
      \[
       \mathcal{S}_{T^{\perp}}(S^{(\mu,0)}\mathcal{Q}\otimes \det(Q)^{-1}) \cong \mathcal{S}_{\mathcal{D}}(S^{(n-k,\mu)}\mathcal{Q}) \quad \text{for all $\mu \succ \lambda$}
      \]
      up to a homological shift, as the second condition of Lemma \ref{LemmaSerreFunctor} is satisfied trivially for the largest $\mu$.  Hence, we still need to verify that for all $\mu \succ \lambda$,
     \begin{equation*}
        \mathbf{R}\text{Hom}(\mathcal{S}_{\mathcal{D}}(S^{(n-k,\mu)}\mathcal{Q}),\mathcal{S}_{\mathcal{D}}(S^{(n-k,\lambda)}\mathcal{Q})) = 0.
     \end{equation*}
     And this is equivalent to showing for all $\mu \succ \lambda$,
     \begin{equation*}
         \mathbf{R}\text{Hom}(S^{(n-k,\mu)}\mathcal{Q},S^{(n-k,\lambda)}\mathcal{Q}) = 0.
     \end{equation*}         
    By Lemma \ref{LemmaVecBund}, it suffices to show for all $i$,
    \begin{equation*}
        H^i(S^{(n-k,\lambda)}\mathcal{Q}\otimes S^{\overline{-(n-k,\mu)}}\mathcal{Q}\otimes\det(\mathcal{R})^{\otimes (n - k)}) = 0.
    \end{equation*}
    Using Littlewood--Richardson rule, we have
    \begin{equation*}
        S^{(n-k,\lambda)}\mathcal{Q}\otimes S^{\overline{-(n-k,\mu)}}\mathcal{Q}\otimes\det(\mathcal{R})^{\otimes(n - k)} = \bigoplus_{\gamma} c^\gamma_{(n-k,\lambda),\overline{-(n-k,\mu)}} S^\gamma\mathcal{Q}\otimes \det(\mathcal{R})^{\otimes (n - k)},
    \end{equation*}
    and $\gamma_k \leq n - k +\lambda_i - \mu_i$. Since $\mu \succ \lambda$, there is at least one $i$ such that $\mu_i > \lambda_i$. Hence $\gamma_k \leq n - k - 1$. Using the same argument as above, we have
    \begin{equation*}
        H^i(S^\gamma\mathcal{Q}\otimes\det(\mathcal{R})^{\otimes (n - k)}) = 0
    \end{equation*}
    for all such $\gamma$ and all $i\geq 0$. Hence
    \begin{equation*}
        \mathbf{R}\text{Hom}(S^{(n-k,\mu)}\mathcal{Q},S^{(n-k,\lambda)}\mathcal{Q}) = 0.
    \end{equation*}

    Thus, by Lemma \ref{LemmaSerreFunctor}, for all $\lambda\in Y(k - 1,n-k)$, we know that $\mathcal{S}_{\mathcal{D}}(S^{(n-k,\lambda)}\mathcal{Q})$ is isomorphic to $\mathcal{S}_{T^{\perp}}(S^{(\lambda,0)}\mathcal{Q}\otimes \det(Q)^{-1})$ up to a homological shift. To conclude that the shift in degree is $k - n$, it suffices to check that for all $\lambda\in Y(k - 1, n - k)$, we have:
    \begin{equation*}
        \mathbf{R}\text{Hom}(S^{(\lambda,0)}\mathcal{Q}\otimes \det(\mathcal{Q})^{-1}, \mathcal{S}_{\mathcal{D}}(S^{(n-k,\lambda)}\mathcal{Q}[k-n])) = \mathbb{F}.
    \end{equation*}
    Using the functoriality of $\mathcal{S}_{\mathcal{D}}$, it is equivalent to showing:
    \begin{equation*}
        \mathbf{R}\text{Hom}(S^{(n - k,\lambda)}\mathcal{Q},S^{(\lambda,0)}\mathcal{Q}\otimes\det(\mathcal{R})[n - k]) = \mathbb{F}.
    \end{equation*}
    By Lemma \ref{LemmaVecBund}, it suffices to show
    \begin{equation*}
        H^{i+ n - k}(S^{(\lambda,0)}\mathcal{Q}\otimes S^{\overline{-(n - k,\lambda)}}\mathcal{Q}\otimes\det(\mathcal{R})^{\otimes (n - k + 1)}) = 
        \begin{cases}
            \mathbb{F}, \text{when } i = 0, \\
            0, \text{otherwise}.
        \end{cases}        
    \end{equation*}
    Using Littlewood--Richardson rule, we have 
    \begin{equation*}
        S^{(\lambda,0)}\mathcal{Q}\otimes S^{\overline{-(n - k,\lambda)}}\mathcal{Q}\otimes\det(\mathcal{R})^{\otimes (n - k + 1)} = \bigoplus_{|\gamma| = (k - 1)(n - k)}c^{\gamma}_{\lambda,\overline{-(n - k, \lambda)}} S^{\gamma}\mathcal{Q}\otimes\det(\mathcal{R})^{\otimes (n - k + 1)}.
    \end{equation*}
    Note that $\gamma_k\leq n - k,\gamma_{k - 1}\leq n - k$. Using the same argument as above, only those $\gamma$ with $\gamma_{k} = 0$ and $\gamma_{k - 1} = n - k$ might have non-trivial cohomology. Since $|\gamma| = (k - 1)(n - k)$, there is only one such possible case: $\gamma = (n - k)^{k - 1}$. Hence
    \begin{equation*}
        H^{i+ n - k}(S^{(\lambda,0)}\mathcal{Q}\otimes S^{\overline{-(n - k,\lambda)}}\mathcal{Q}\otimes\det(\mathcal{R})^{\otimes (n - k + 1)}) = H^{i + n - k}\left(c^{(n - k)^{(k - 1)}}_{\lambda,\overline{-(n - k,\lambda)}}S^{(n - k)^{k - 1}}\mathcal{Q}\otimes \det(\mathcal{R})^{\otimes (n - k + 1)}\right).
    \end{equation*}
    By the Borel--Weil--Bott theorem, there is at most one $i$ such that 
    \begin{equation*}
        H^{i + n - k}\left(c^{(n - k)^{(k - 1)}}_{\lambda,\overline{-(n - k,\lambda)}}S^{(n - k)^{k - 1}}\mathcal{Q}\otimes \det(\mathcal{R})^{\otimes (n - k + 1)}\right)\neq 0. 
    \end{equation*}
    Define $\sigma_0 \coloneqq (k+1,k)(k+2,k+1)\cdots(n,n-1)$. Then $l(\sigma) = n - k$, and
    \begin{equation*}
        ((n-k)^{k - 1}, 0 , (n - k+1)^{n - k})\cdot\sigma_{0} = (n - k)^{n},
    \end{equation*}
    hence 
    \begin{align*}
        H^{ n - k}\left(c^{(n - k)^{(k - 1)}}_{\lambda,\overline{-(n - k,\lambda)}}S^{(n - k)^{k - 1}}\mathcal{Q}\otimes \det(\mathcal{R})^{\otimes (n - k + 1)}\right) &= (S^{(n-k)^{(n)}}(\mathbb{F}^n)^\vee)^{\oplus c^{(n - k)^{(k - 1)}}_{\lambda,\overline{-(n - k,\lambda)}}} \\
        &=[\det((\mathbb{F}^n)^\vee)^{\otimes (n - k)}]^{\oplus c^{(n - k)^{(k - 1)}}_{\lambda,\overline{-(n - k,\lambda)}}} = \mathbb{F}^{c^{(n - k)^{(k - 1)}}_{\lambda,\overline{-(n - k,\lambda)}}}.
    \end{align*}
    From Lemma \ref{finalsteplemma}, we know that $c^{(n - k)^{(k - 1)}}_{\lambda,\overline{-(n - k,\lambda)}} = 1$. Hence 
    \begin{equation*}
        \mathbf{R}\text{Hom}(S^{(\lambda,0)}\mathcal{Q}\otimes \det(\mathcal{Q})^{-1}, \mathcal{S}_{\mathcal{D}}(S^{(n-k,\lambda)}\mathcal{Q}[k-n])) = \mathbb{F}.
    \end{equation*}
    And we may conclude.
\end{proof}

\begin{remark}
    We want to point out that, in a slightly different form, the same mutation is computed explicitly in \cite[Proposition 5.3]{fonarev2013minimal}.
\end{remark}

Theorem \ref{ThmMutation} allows us to prove Corollary \ref{cor:miracle}, which is needed to prove Theorem \ref{thm:split}. 

\begin{remark}\label{AdjointFunctor}
    Recall that we have adjunctions $i^*\dashv i_*$, $i_* \dashv i^! = i^*(-) \otimes \det (\mathcal{Q}_{k-1,n-1})[-n + k ]$. 
    Moreover, if we denote by $f: T^\perp\hookrightarrow \mathcal{D}$ the inclusion, by property of Serre functors, the adjunction 
   $L_T\dashv f$ induces an adjunction $\mathcal{S}^{-1}_{\mathcal{D}}\circ f \circ \mathcal{S}_{T^{\perp}} \dashv L_T$. All of these adjunctions will be used in the proof of Corollary \ref{cor:miracle}.
\end{remark}

In the following, we omit the decorations on $\mathcal{Q}$; the Grassmannian to which $\mathcal{Q}$ belongs should be clear from context.

\begin{corollary} \label[corollary]{cor:miracle}
    $L_T i_* S^{\lambda} \mathcal{Q} = S^{(\lambda,0)}\mathcal{Q}\otimes \det (\mathcal{Q})^{-1} [n - k]$.
\end{corollary}




\begin{proof}
By \cite[Remark 0.4]{kapranov1985derived}, we have an isomorphism 
\[
S^{\lambda}(\mathcal{Q}\oplus \mathcal{O}) \cong \bigoplus_{\lambda'}S^{\lambda'}\mathcal{Q},
\]
where the summation is over all Young diagrams $\lambda'$ such that $\lambda_{1} \geq \lambda_{1}'\geq \lambda_{2}\geq \lambda_{2}'\geq \cdots$. We will denote this property by $\lambda' \trianglelefteq \lambda$. Note that $\lambda' \trianglelefteq \lambda \Rightarrow \lambda' \preceq \lambda$. 

Recall that $T^{\perp}$ has a full exceptional collection $\left\langle S^{(\lambda,0)}\mathcal{Q}\otimes \det (\mathcal{Q})^{-1} | \lambda \in Y(k-1,n-k)  \right\rangle$. Therefore, to prove that $L_T i_* S^{\lambda} \mathcal{Q} = S^{(\lambda,0)}\mathcal{Q}\otimes \det (\mathcal{Q})^{-1} [n - k]$, it suffices to prove 
\begin{enumerate}[(i)]
        \item $\mathbf{R}\text{Hom} (L_T i_*S^{\lambda}\mathcal{Q}, S^{(\mu,0)}\mathcal{Q}\otimes (\det(\mathcal{Q}))^{-1}[n-k]) = 0 $ for all $\mu\in Y(k-1,n-k)$ with $\mu \prec \lambda$;
        \item $\mathbf{R}\text{Hom} (S^{(\mu,0)}\mathcal{Q}\otimes (\det(\mathcal{Q}))^{-1}[n-k], L_T i_*S^{\lambda}\mathcal{Q}) = 0 $ for all $\mu\in Y(k-1,n-k)$ with $\lambda \prec \mu$;
        \item $\mathbf{R}\text{Hom} (L_T i_* S^{\lambda}\mathcal{Q}, S^{(\lambda,0)}\mathcal{Q}\otimes (\det (\mathcal{Q}))^{-1}[n-k]) = \mathbb{F}$.
    \end{enumerate}
    \begin{enumerate}[(i)]
        \item 
       We compute 
        \begin{align*}
            \mathbf{R}\text{Hom} (L_T i_*S^{\lambda}\mathcal{Q}, S^{(\mu,0)}\mathcal{Q}\otimes \det(\mathcal{Q})^{-1}[n-k]) &\cong \mathbf{R}\text{Hom}(i_* S^{\lambda}\mathcal{Q}, S^{(\mu,0)}\mathcal{Q}\otimes \det(\mathcal{Q})^{-1}[n-k]) \\
            &\cong \mathbf{R}\text{Hom}( S^{\lambda}\mathcal{Q}, i^!(S^{(\mu,0)}\mathcal{Q}\otimes \det(\mathcal{Q})^{-1}[n-k])) \\
            &\cong \mathbf{R}\text{Hom}( S^{\lambda}\mathcal{Q}, i^*S^{(\mu,0)}\mathcal{Q} \otimes \det(\mathcal{Q})^{-1}[n-k] \otimes \det(\mathcal{Q})[-n+k] ) \\
             &\cong \mathbf{R}\text{Hom}( S^{\lambda}\mathcal{Q}, i^{*}S^{(\mu,0)}\mathcal{Q})\\
             &\cong \mathbf{R}\text{Hom}\left( S^{\lambda}\mathcal{Q}, \bigoplus_{\mu' \trianglelefteq \mu} S^{\mu'}\mathcal{Q}\right) = 0,
        \end{align*}
        where the final line is zero because $\mu' \trianglelefteq \mu \Rightarrow \mu' \preceq \mu \prec \lambda$.

        \item Again, we compute
        \begin{align*}
            \mathbf{R}\text{Hom} (S^{(\mu,0)}\mathcal{Q}\otimes (\det(\mathcal{Q}))^{-1}[n-k], L_T i_*S^{\lambda}\mathcal{Q}) &\cong \mathbf{R}\text{Hom} (\mathcal{S}^{-1}_{\mathcal{D}}\circ f \circ \mathcal{S}_{T^{\perp}}(S^{(\mu,0)}\mathcal{Q}\otimes (\det(\mathcal{Q}))^{-1}[n-k]),i_*S^{\lambda}\mathcal{Q})\\
            &\cong\mathbf{R}\text{Hom} (S^{(n-k,\mu)}\mathcal{Q}, i_*S^{\lambda}\mathcal{Q}) \\
            &\cong\mathbf{R}\text{Hom} (i^{*}S^{(n-k,\mu)}\mathcal{Q},S^{\lambda}\mathcal{Q}) \\
            &\cong\mathbf{R}\text{Hom} \left(\bigoplus_{\nu \trianglelefteq (n-k,\mu)} S^{\nu}\mathcal{Q}, S^{\lambda}\mathcal{Q}\right),
        \end{align*}
       where the second line follows from Theorem \ref{ThmMutation}. Note that, if $\nu = (\nu_{1}, \dots, \nu_{k-1})$, then the condition $\nu \trianglelefteq (n-k,\mu)$ means $n-k \geq \nu_{1} \geq \mu_{1} \geq \nu_{2} \geq \mu_{2}\geq \cdots$. Therefore, $\lambda \prec \mu \preceq \nu$ and the last line above is equivalent to zero.
        
        \item This computation is similar to i), but in the final line, only the term $\mathbf{R}\text{Hom}(S^{\lambda}\mathcal{Q}, S^{\lambda}\mathcal{Q})$ will remain. This is $\mathbb{F}$.
    \end{enumerate}

\end{proof}

\begin{remark} \label[remark]{remark:mutationbasechange}
For a divisorial scheme $S$, define $T_{S}\coloneqq \left\langle S^{\lambda}\mathcal{Q} \otimes p^{*}D^{b}\text{Vect}(S)\right\rangle_{\lambda \in Y(k,n - k - 1)}$ and $T_{\mathbb{F}}\coloneqq T_{\Spec(\mathbb{F})}$. Then, the corresponding statements for divisorial schemes over a field of characteristic zero follow: for $E\in D^{b}\text{Vect}(S)$,
\begin{align*}
    L_{T_S} (\phi^* S^{(n-k,\lambda)}\mathcal{Q} \otimes p^* E) &\simeq \phi^* L_{T_{\mathbb{F}}}\left(S^{(n-k,\lambda)}\mathcal{Q}\right) \otimes p^* E \\
     &\simeq \phi^*\left(S^{(\lambda,0)}\mathcal{Q} \otimes \det(\mathcal{Q})^{-1}[n-k]\right) \otimes p^*E \\
     &\simeq S^{(\lambda,0)}\mathcal{Q} \otimes \det(\mathcal{Q})^{-1}[n-k] \otimes p^*E,
\end{align*}
and
\begin{align*}
    L_{T_{S}}i_{*}(\phi^{*}S^{\lambda} \mathcal{Q} \otimes p^{*}E) &\simeq L_{T_{S}}\left(i_{*}\phi^{*}S^{\lambda}\mathcal{Q}\otimes p^{*}E \right) \quad (\text{projection formula and $p^{*} \simeq i^{*}p^{*}$}) \\ 
    & \simeq L_{T_{S}} \left(\phi^{*}i_{*}S^{\lambda}\mathcal{Q}\otimes p^{*}E \right) \quad (\text{this is true as $\phi^{*}i_{*}\mathcal{O} \xrightarrow{\simeq} i_{*}\phi^{*}\mathcal{O}$})  \\
    & \simeq \phi^{*}\left(L_{T_{\mathbb{F}}}i_{*}S^{\lambda}\mathcal{Q}\right)\otimes p^{*}E\quad \\
    & \simeq \phi^{*}\left(S^{(\lambda,0)}\mathcal{Q}\otimes \det(\mathcal{Q})^{-1}[n-k]\right)\otimes p^{*}E \\
     & \simeq \left(S^{(\lambda,0)}\mathcal{Q}\otimes \det(\mathcal{Q})^{-1}[n-k]\right)\otimes p^{*}E.
\end{align*}
Here, the mutation commutes with base-change because $\phi^* S^{(n-k,\lambda)}\mathcal{Q} \otimes p^* E $  and $\phi^{*}i_{*}S^{\lambda}  \mathcal{Q} \otimes p^{*}E$ sit inside two distinguished triangles associated with the same semi-orthogonal decomposition. 

In formulae: let $\square \in \{S, \mathbb{F}\}$, and let $ \pi_{T_{\square}}: D^{b}\text{Vect}(\text{Gr}^{k,n}_{\square}) \rightarrow T_{\square}, L_{T_{\square}}: D^{b}\text{Vect}(\text{Gr}^{k,n}_{\square}) \rightarrow T^{\perp}_{\square} $ denote the projection and mutation functors of $T_{\square}$ with respect to the semi-orthogonal decomposition $D^{b}\text{Vect}(\text{Gr}^{k,n}_{\square}) = \langle T^{\perp}_{\square}, T_{\square} \rangle$. Then, for any $\mathcal{G} \in D^{b}\text{Vect}(\text{Gr}^{k,n}_{\mathbb{F}})$ and any $\mathcal{H} \in D^{b}\text{Vect}(S)$ we have a diagram of distinguished triangles
\[
\begin{tikzcd} 
(\phi^{*}\pi_{T_{\mathbb{F}}}\mathcal{G}) \otimes p^{*}(\mathcal{H}) \arrow[d,"\cong" ] \arrow[r]&  (\phi^{*}\mathcal{G}) \otimes p^{*}(\mathcal{H})\arrow[d, "="] \arrow[r] &  (\phi^{*}L_{T_{\mathbb{F}}}\mathcal{G}) \otimes p^{*}(\mathcal{H})\arrow[d, "\cong"] \\
\pi_{T_{S}}(\phi^{*}\mathcal{G} \otimes p^{*}(\mathcal{H}))  \arrow[r]  & (\phi^{*}\mathcal{G}) \otimes p^{*}(\mathcal{H})  \arrow[r]  &  L_{T_{\mathbb{S}}}(\phi^{*}\mathcal{G} \otimes p^{*}(\mathcal{H}))
\end{tikzcd},
\]
where the objects on the left lie in $T_{S}$ and the objects on the right lie in $T_{S}^{\perp}$. Therefore, we in particular have an isomorphism $L_{T_{\mathbb{S}}}(\phi^{*}\mathcal{G} \otimes p^{*}(\mathcal{H})) \cong (\phi^{*}L_{T_{\mathbb{F}}}\mathcal{G}) \otimes p^{*}(\mathcal{H})$.
\end{remark}

In the proof of Theorem \ref{thm:compute}, we also need the following computation. We include it here because the proof uses the Borel--Weil--Bott Theorem and similar arguments to the above. 

\begin{lemma} \label{lemma:0end}
    In the above notation, we have $H^{i}\left(\text{Gr}^{k,n}_{\mathbb{F}}, (S^{\lambda}\mathcal{Q})^{\vee} \otimes S^{\lambda}\mathcal{Q} \right) = 0$ for all $i > 0$.
\end{lemma}

\begin{proof}
    Let $\lambda = (\lambda_{1}, \dots, \lambda_{k})$. Then, 
    \[
    (S^{\lambda}\mathcal{Q})^{\vee}\otimes S^{\lambda}\mathcal{Q} \cong \bigoplus_{\gamma} c^{\gamma}_{\lambda,\overline{-\lambda}} S^{\gamma}\mathcal{Q} \otimes \det\mathcal{R}^{\otimes \lambda_{1}}. 
    \]
    Note, $|\gamma| = |\overline{-\lambda}| + |\lambda| = k\lambda_{1}$. Furthermore, we have $\gamma_{k} \leq \lambda_{1}$. Therefore, there are two cases to consider: either $\gamma_{k} < \lambda_{1}$ or $\gamma_{k} = \lambda_{1}$.  

    In the first case, we can permute the element $\gamma_{k}$ to the right in the sequence $(\gamma, (\lambda_{1})^{n-k})$ by a series of adjacent transpositions to eventually get a sequence of the form $(\dots, \lambda_{1}-1, \lambda_{1}, \dots)$. Therefore, by the Borel--Weil--Bott Theorem, all cohomology groups are zero in this case. 

    In the second case, as $|\gamma| = k\lambda_{1}$ and $\gamma_{1}\geq \dots \geq \gamma_{k}$, hence $\gamma = (\lambda_{1},\dots , \lambda_{1})$. Therefore, in the Borel--Weil--Bott Theorem, we need to consider the sequence $(\lambda_{1})^{n}$. By this theorem, the identity permutation is the unique permutation such that $(\lambda_{1})^{n} \cdot \text{id} = (\lambda_{1})^{n}$ is a partition, so that we deduce all higher cohomology groups are zero in this case also. 
\end{proof}

\appendix

\section{Push-forward in Hermitian $K$-theory,  by M. Schlichting} \label{appendix:marco}
\vspace{2ex}

\subsection{Introduction}

Let $X$ be a divisorial scheme, a.k.a. a scheme with an ample family of line bundles \cite[Definition 2.11]{TT}.
For instance, $X$ can be quasi-projective over an affine scheme such as $\pp^n_R$ or $\Gr_R(d,n)$, or $X$ can be regular noetherian and separated.
Let $i:Z\subset X$ be a regular embedding of codimension $c$. 
So, the ideal sheaf $\I \subset O_X$ of the embedding is locally generated by a regular sequence of length $c$. 
Denote by $\omega_i = \det(\I/\I^2)^{-1}$ the dual of the determinant bundle of the conormal sheaf $\I/\I^2$.
The goal of this appendix is to construct push forward maps 
$$Ri_*: GW^{[r-c]}(Z,\omega_i\cdot i^*L) \to GW^{[r]}(X, L),\hspace{3ex} Ri_*: \GW^{[r-c]}(Z,\omega_i\cdot i^*L) \to \GW^{[r]}(X, L)$$
 of Grothendieck--Witt and Karoubi--Grothendieck--Witt spectra such that $j^*Ri_*=0$ for $j:U =X-Z \subset X$ the open embedding of the complement.
 This is used in the proof of Theorem \ref{thm:split} in the body of the text.
The maps will be induced by (a zigzag of) maps of complicial exact (or dg) categories with dualities and weak equivalences 
$$(\sPerf(Z),\omega_i [-c]) \to (\sPerf(X),O_X)$$
between the usual categories of strictly perfect complexes on $Z$ and $X$ which is the usual $Ri_*$ upon taking derived categories.
The point is to make the constructions on the level of complicial exact (or dg) categories and to give explicit duality compatibility maps.

Our construction is elementary in that we do not use infinity categories nor the existence of a right adjoint to $Ri_*$  as in \cite{calmes2024motivic} nor do we use Grothendieck duality as in \cite{huang2023connecting}.
Instead, we deal with the codimension $1$ case first via a dg resolution and the higher codimension case is reduced to codimension $1$ via blow ups.
As this is not needed in the body of the text, we have refrained from giving details as to functoriality and various base-change and excess intersection formulas, though functoriality of push-forward for the maps in a blow-up square are immediate from our construction; see Remark \ref{rmk:FunctorialSquarePushGW}.

The conscientious reader will notice that our schemes ought to be defined over $\Z[1/2]$ in order to apply $\GW$ and $GW$ as defined in \cite{schlichting17}. 
However, our constructions also work when $GW$ and $\GW$ are understood as the symmetric Grothendieck--Witt and Karoubi--Grothendieck--Witt theories in the sense of \cite{calmes2023hermitian}.
\vspace{1ex}

\noindent{\em Notations and conventions}.
All our schemes are assumed to be divisorial.
All our complicial exact (or dg) categories will be strictly perfect complexes $\sPerf(\A)$ on a (commutative dg) scheme $(X,\A)$; see \S\ref{sec:sPerfDG}.
Such categories are closed symmetric monoidal, and all dualities will be given by the internal hom into an object $L$ and equipped with the canonical double dual identification $E \to Hom(Hom(E,L),L)$ adjoint to the evaluation $Hom(E,L) \otimes E \to L$.
We specify such categories with weak equivalences and duality simply by $(\sPerf(\A),L)$ where the weak equivalences are understood to be the quasi-isomorphisms.
If $\T \subset \D\sPerf(\A)$ is a full triangulated subcategory of the derived category of $\sPerf(\A)$ then 
$\left(\frac{\sPerf(\A)}{\T},L\right)$ will mean the complicial exact (or dg) category $\sPerf(\A)$ with duality in $L$ where the weak equivalences are the maps whose cones are retracts of objects in $\T$.
The tensor product of line bundles $L$ and $M$ will be denoted by $L\cdot M$.

\subsection{Push forward for codimension $1$ regular embeddings}

Let $i:Z \subset X$ be a regular embedding of codimension $1$ with $X$ divisorial.
The ideal sheaf $\I = \ker (O_X \to i_*O_Z)$ of the embedding is a line bundle on $X$.
Then $\A_Z = (\I \hookrightarrow O_X)$ defines a  quasi-coherent differential graded algebra (dga) on $X$, where $O_X$ sits in degree zero.
The quotient map $O_X \to i_*Q_Z$ induces a quasi-isomorphism of dga's $f: \A_Z \to i_*O_Z$.
A dg $\A_Z$-module $M$ is strictly perfect if $X$ has a covering by open affines $U_i \subset X$ such that $M(U_i)$ is a finite cell $\A_Z(U_i)$-module, that is, it has a finite filtration by dg $\A_Z(U_i)$-submodules with successive quotients isomorphic to finite direct sums of $\A_Z(U_i)[r]$ for various $r\in \Z$.
A homomorphism $M \to N$ of dg $\A_Z$-modules is a quasi-isomorphism if it is an isomorphism on all homology sheaves; see \S\ref{sec:sPerfDG} for more details.

Since $\sPerf(i_*O_Z) = \sPerf(Z)$, we obtain a map of complicial exact categories $f^*:\sPerf(\A_Z) \to \sPerf(Z):M \mapsto i_*O_Z\otimes_{\A_Z}M$ which induces an equivalence of derived categories (after idempotent completion), by Corollary \ref{cor:sPerfEq}:
\begin{equation}
\label{eqn:KPushCod1}
f^*:\widetilde{\D}\sPerf(\A_Z) \stackrel{\simeq}{\longrightarrow} \D\sPerf(Z).
\end{equation}
Since every strictly perfect $\A_Z$-module is a strictly perfect $O_X$-module, the restriction of scalars along the ring homomorphism $O_X \to \A_Z$ induces a dg functor $\fgt:\sPerf(\A_Z) \to \sPerf(X)$.
Combining the two functors, the functor $Ri_*$ can be represented as the composition of functors
\begin{equation}
\label{eqn:PushCod1}
\xymatrix{\sPerf(Z) & \sPerf(\A_Z) \ar[l]^{\sim}_{f^*} \ar[r]^{\fgt} & \sPerf(X)}
\end{equation}
where $\stackrel{\sim}{\to}$ means equivalence of derived categories after idempotent completion.

We will equip the categories in (\ref{eqn:PushCod1}) with dualities that are preserved by the functors between them.
Let $\A$ be a commutative differential graded quasi-coherent $O_X$-algebra.
We make the category $\sPerf(\A)$ of strictly perfect dg $\A$-modules into a complicial exact category with weak equivalences and duality as follows.
The dual $E^{\vee}$ of a strictly perfect dg $\A$-module $E$ is the dg $\A$-module
$$E^{\vee} = Hom_{\A}(E,\A).$$
Locally, if $E(U)$ is a finite cell $A(U)$-algebra, then $E^{\vee}(U) = Hom_{\A(U)}(E(U),\A(U))$ is a finite cell $\A(U)$-algebra. 
See \cite[\S III.4]{KrizMay}.
In particular, $E^{\vee}$ is strictly perfect if $E$ is, and $(\vee, \can)$ defines a duality on $\sPerf(\A)$.
The duality preserves quasi-isomorphisms. 
Hence, we obtain a complicial exact category with weak equivalences and duality
$$\sPerf(\A)= (\sPerf(\A),\quis,\vee,\can).$$
More generally, if $L$ is a line bundle on $X$ and $i\in \Z$ an integer, then
$$E^{\vee_{L[i]}} = Hom_{\A}(E,\A\otimes_{O_X} L[i])$$
defines a duality on $\sPerf(\A)$, and we obtain the 
complicial exact category with weak equivalences and duality
$$(\sPerf(\A), L[i]).$$

\begin{definition}
\label{dfn:GWpushCod1}
Let $i:Z \subset X$ be a codimension $1$ regular embedding of divisorial schemes.
We make the restriction of scalars $\fgt:\sPerf(\A_Z) \to \sPerf(X)$ into a form functor
$$\fgt: (\sPerf(\A_Z),L) \to (\sPerf(X), \I\cdot L[1])$$
with duality compatibility map 
$$Hom_{\A_Z}(E,\A_Z\otimes L) \to Hom_{O_X}(E,\I\otimes L[1])$$
induced by the map of complexes $\A_Z \to \I[1]$ that is the projection onto the degree $1$ part.
This map is a quasi-isomorphism as it is locally so.
In this way, we obtain a push-forward map 
$$Ri_*:\GW^{[n]}(Z,i^*L) \to \GW^{[n+1]}(X,\I_Z \cdot L)$$  by applying the functor $\GW$ to the composition
\begin{equation}
\label{eqn:GWPushCod1}
(\sPerf(Z), i^*L[n]) \stackrel{\sim}{\leftarrow} (\sPerf(\A_Z), L[n]) \stackrel{\fgt}{\longrightarrow} (\sPerf(X), \I\otimes L[n+1])
\end{equation}
since $\GW$ sends the left map labelled $\stackrel{\simeq}{\leftarrow}$ to a weak equivalence of spectra \cite[Theorem 8.9]{schlichting17}.
Replacing $L$ with $\I_Z^{-1}\cdot L$ yields the pushforward map
$$Ri_*:\GW^{[n]}(Z,\omega_i\cdot i^*L) \to \GW^{[n+1]}(X, L).$$
\end{definition}

\begin{lemma}
\label{lem:Cod1Uzero}
Let $i:Z \to X$ be a codimension $1$ regular embedding of divisorial schemes and $j_U:U=X-Z \to X$ the embedding of the open complement.
Then $j_U^*\circ Ri_*$ has image in the acyclic complexes of $\sPerf(U)$.
In particular, the following composition is zero in the homotopy category of spectra
$$0=j_U^*\circ Ri_*: GW^{[r-1]}(Z,\omega_i\cdot i^*L) \stackrel{Ri_*}{\longrightarrow} GW^{[r]}(X,L) \stackrel{j_U^*}{\longrightarrow} GW^{[r]}(U,j^*L).$$
\end{lemma}

\begin{proof}
Denote by ${\A_Z}_{|U}$ be the restriction of the sheaf of dga's $\A_Z$ to $U$.
The claim follows from the commutative diagram
$$\xymatrix{
& \sPerf(\A_Z) \ar[dl]_{\sim} \ar[r]^{j_U^*}  \ar[d]^{\fgt} & \sPerf({\A_Z}_{|U}) \ar[d]^{\fgt} \\
\sPerf(Z) \ar[r]_{Ri_*} & \sPerf(X) \ar[r]_{j_U^*} & \sPerf(U)}$$
since ${\A_Z}_{|U}$ is a sheaf of contractible dga's.
In particular, any object of $\sPerf({\A_Z}_{|U})$ is acyclic.
\end{proof}

\subsection{Push forward for projective bundle projections}

Let $Y$ be a divisorial scheme and $\E$ a vector bundle on $Y$ of rank $r+1$.
Denote by $q:\pp\E \to Y$ the  projection.
The full triangulated subcategory 
$\ker Rq_*$ of $\D\sPerf(\pp\E)$ is generated by $Lq^*\sPerf(Y)\otimes O(-i)$ for $i=1,...,r$ \cite{Thomason:Eclate}, \cite[Lemmas 1.2, 1.4]{CHSW}.
The functor $Rq_*$ factors through $\D\sPerf(Y)/\ker Rq_*$ and induces an equivalence of triangulated categories 
$Rq_*: \D\sPerf(\pp\E)/\ker Rq_* \stackrel{\sim}{\longrightarrow} \D\sPerf(Y)$ with inverse $Lq^*$ \cite[Lemma 1.4]{CHSW}.
We can thus represent the functor $Rq_*$ as the composition of functors
\begin{equation}
\label{eq:Rq*}
Rq_*: \xymatrix{ \sPerf(\pp\E) \ar[r]^{\hspace{-1ex}\text{Loc}} & \frac{\sPerf(\pp\E)}{\ker Rq_*} & \sPerf(Y). \ar[l]^{\hspace{1ex}\sim}_{\hspace{1ex}Lq^*}}
\end{equation}

We will equip the categories in (\ref{eq:Rq*}) with dualities that are preserved by the functors between them.
Recall the sheaf of K\"ahler differentials $\Omega_q$ of $q:\pp\E \to Y$.
It is $\Omega_q = \ker(\mu_{\E})$, where $\mu_{\E}$ is the multiplication map $\mu_{\E}:O_{\pp\E}(-1)\otimes \E \to O_{\pp\E}$.
Denote by $\Kosz(\mu_{\E})$ the Koszul complex of $\mu_{\E}$.
This is an acyclic complex concentrated in degrees $0,...,r+1$.
Its degree $r+1$ and degree $0$ parts are 
$$\Kosz_{r+1}(\mu_{\E}) = \Lambda^{r+1}(\E(-1)) =(\det\E)(-r-1) = \omega_q,\hspace{3ex}\Kosz_{0}(\mu_{\E}) = O_{\pp\E}.$$
Denote by $\Kosz_{\leq r}(\mu_{\E})$ the full subcomplex of $\Kosz_{\leq r}(\mu_{\E})$ concentrated in degrees $0,...,r$.
The degree $r+1$ differential of $\Kosz(\mu_{\E})$ defines a quasi-isomorphism of complexes 
$$\alpha: \omega_q[r] \stackrel{\sim}{\longrightarrow} \Kosz_{\leq r}(\mu_{\E}).$$
The degree $0$ inclusion 
$$\beta: O_{\pp\E} \to \Kosz_{\leq r}(\mu_{\E})$$ has cone in the 
subcategory 
$ \ker Rq_*$. 
In particular, it is a weak equivalence in the complicial exact category with weak equivalences $\frac{\sPerf(\pp\E)}{\ker Rq_*}$.
Since the subcategory $\ker Rq_*$ is closed under the duality with values in $\omega_q = \det \Omega_q = O_{\pp\E}(-r-1)\det\E$, we obtain a diagram of complicial exact categories with weak equivalences and duality
\begin{equation}
\label{eq:Rq*v2}
 \xymatrix{ \left(\sPerf(\pp\E),\omega_q[r]\right) \ar[r]^{\hspace{-4ex}(\text{Loc},\alpha)} & \left( \frac{\sPerf(\pp\E)}{\ker Rq_*},\Kosz_{\leq r}(\mu_{\E})\right) & \left(\sPerf(Y) \ar[l]^{\hspace{5ex}\sim}_{\hspace{5ex}(q^*,\beta)},O_Y\right)},
\end{equation}
where the arrow in the wrong direction is an equivalence of associated derived categories.
Applying Karoubi-Grothendieck-Witt spectra, this yields a well-defined map in the homotopy category of spectra
$$Rq_* = (q^*,\beta)^{-1}\circ (\text{Loc},\alpha): \GW^{[n+r]}(\pp\E,q^*L\cdot \omega_q) \longrightarrow \GW^{[n]}(Y,L).$$

\subsection{Push forward for regular embeddings}

Let $X$ be a divisorial scheme, and let $i:Y \to X$ be a regular embedding of pure codimension $c$, and consider the blow-up square
\begin{equation}
\label{eqn:BlowUpSquare}
\xymatrix{Y' \ar[r]^{j} \ar[d]_q & X'\ar[d]^p\\ Y \ar[r]_i & X.}
\end{equation}
The exceptional divisor inclusion $j:Y' \to X$ is a regular embedding of codimension $1$ with ideal sheaf $O_{X'}(1)$.
The full triangulated subcategory
$\ker(Rp_*) \subset \D\sPerf(X')$
is generated by the subcategories $Rj_*Lq^*\D\sPerf(Y)\otimes O_{X'}(-i)$,  $i=1,...,c-1$; see \cite{Thomason:Eclate}, \cite[Lemma 1.2]{CHSW}.

\begin{lemma}
The following diagram commutes up to natural isomorphism of functors 
\begin{equation}
\label{eqn:BlowUpPushD}
\xymatrix{ 
\D\sPerf(Y)  \ar[r]^{Lq^*} \ar[d]_{Ri_*} & \D\sPerf(Y') \ar[r]^{Rj_*} &\D\sPerf(X') \ar[d]^{\text{Loc}}  \\
 \D\sPerf(X) \ar[r]_{\hspace{-1ex}Lp^*}  & \D\sPerf(X') \ar[r]_{\hspace{-5ex}\text{Loc}} & \D\sPerf(X')/\ker Rp_*}
\end{equation}
\end{lemma}

\begin{proof}
The unit of adjunctions $1_Y \cong Rq_*Lq^*$ and $1_X \cong Rp_*Lp^*$ are isomorphisms, and the co-unit of adjunction $Lp^*Rp_* \to 1_{X'}$ has cone in $\ker Rp_*$.
Using the natural isomorphism $Ri_*Rq_*\cong Rp_*Rj_*$, we obtain the sequence of natural isomorphisms of functors
$$\xymatrix{
Lp^*Ri_* \ar[r]^{\hspace{-5ex}\cong}_{\hspace{-5ex}\text{unit}} & Rp^*Ri_*Rq_*Lq^* \ar@{=}[r]^{\sim} & Lp^*Rp_*Rj_*Lq^* \ar[rr]^{\hspace{2ex}\mod\ker Rp_*}_{\cong\hspace{1ex}\text{co-unit}} && Rj_*Lq^*.}$$
\end{proof}

Since the composition of the bottom two horizontal functors in (\ref{eqn:BlowUpPushD}) is an equivalence \cite{Thomason:Eclate}, \cite[Lemma 1.4]{CHSW}, we can represent the functor $Ri_*$ as the composition of functors
\begin{equation}
\label{eqn:BlowUpPush}
\sPerf(Y) \stackrel{q^*}{\to} \sPerf(Y') \stackrel{Rj_*}{\to} \sPerf(X') \stackrel{\text{Loc}}{\longrightarrow} \frac{\sPerf(X')}{\ker Rp_*} \stackrel{\hspace{1ex}p^*}{\leftarrow}\sPerf(X),
\end{equation}
where the arrow in the wrong direction is an equivalence on derived categories.
In other words, $Ri_*$ is determined by the commutativity of the diagram
$$\xymatrix{
\sPerf(Y) \ar[r]^{q^*}_{\sim} \ar[d]_{Ri_*} & \frac{ \sPerf(Y')}{\ker Rq_*} \ar[d]_{Rj_*} & \frac{\sPerf (\A_{Y'})}{\ker Rq_*}  \ar[dl]^{\fgt}  \ar[l]_{\sim}\\
\sPerf(X) \ar[r]_{p^*}^{\sim} &  \frac{ \sPerf(X')}{\ker Rp_*} & }.$$

Our task is to equip the categories in (\ref{eqn:BlowUpPush}) with dualities that are preserved by the given functors.

\begin{lemma}
\label{lem:DualIdentnu}
The twist $\nu(i) = \nu\otimes O_{X'}(i)$ of the ideal sheaf inclusion $\nu: O_{X'}(1) \to O_{X'}$ has cone in $\ker(Rp_*)$ for $i=-1,...-c+1$.
In particular, the composition 
$$\gamma: O_{X'} \stackrel{\nu(-1)}{\longrightarrow} O_{X'}(-1) \stackrel{\nu(-2)}{\longrightarrow} \cdots \stackrel{\nu(-c+1)}{\longrightarrow} O_{X'}(-c+1)$$
has cone in $\ker(Rp_*)$.
\end{lemma}

\begin{proof}
By definition, the cone of $\nu$ is $Rj_*O_{Y'}$.
For  $i=-1,..., -c+1$ we then have
$$Rp_*\cone(\nu\otimes O_{X'}(i)) = Rp_*Rj_*O_{Y'}(i) = Ri_*Rq_*(O_{Y'}(i)) =0$$
since $O_{Y'}(i)\in \ker Rq_*$ for $i=-1,...,-c+1$.
\end{proof}

Recall that $\omega_i=(\det\mathcal{N})^{-1}$ where $\mathcal{N}=\I/\I^2$ is the conormal sheaf on $Y$ of the regular embedding $i:Y \to X$ defined by the sheaf of ideals $\I\subset O_X$.
In particular, we have a canonical isomorphism $\can:\omega_q \cdot q^*\omega_i\cong O_{Y'}(-c)$ given by multiplication.
This allows us to define the following zigzag of form functors

\begin{equation}
\label{eqn:Ri*Zigzag}
\xymatrix{
(\sPerf(Y),\omega_i[-c]) \ar[r]^{\hspace{-9ex}(q^*,\beta)} &
 \left(\frac{\sPerf(Y')}{\ker Rq_*},\Kosz_{\leq c-1}(\mu_{\mathcal{N}}) \cdot q^*\omega_i[-c]\right) & \\
&
\left(\frac{\sPerf(Y')}{\ker Rq_*},\omega_q\cdot q^*\omega_i[-1]\right) \ar[u]_{(1,\alpha)}^{\simeq} \ar[d]^{(1,\can)}_{\cong} &  \\
&\left(\frac{\sPerf(Y')}{\ker Rq_*},O_{Y'}(-c)[-1]\right)  \ar[d]^{Rj_*}_{\text{(\ref{eqn:GWPushCod1})}} & \\
& \left(\frac{\sPerf(X')}{\ker Rp_*},O_{X'}(-c+1)\right) & (\sPerf(X),O_X). \ar[l]_{\hspace{5ex}(p^*,\gamma)}^{\hspace{5ex}\simeq}
}
\end{equation}

\begin{definition}
Let $i:Y \to X$ be a codimension $c$ regular embedding of divisorial schemes.
For a line bundle $L$ on $X$ and integer $r\in \Z$, the push-forward map
$$Ri_*: \GW^{[r-c]}(Y,i^*L\cdot \omega_i) \longrightarrow \GW^{[r]}(X,L)$$
is obtained by applying $\GW$ to (\ref{eqn:Ri*Zigzag}) and noting that the arrows in the wrong direction yield isomorphisms in the homotopy category of spectra.
\end{definition}

\begin{lemma} \label{lemma:zerocomp}
Let $i:Y \to X$ be a codimension $c$ regular embedding of divisorial schemes and $j_U:U=X-Y \to X$ the open embedding of the complement.
Then the following composition is zero
$$0=j_U^*\circ Ri_*: \GW^{[r-c]}(Y,i^*L\cdot \omega_i) \stackrel{Ri_*}{\longrightarrow} \GW^{[r]}(X,L) \stackrel{j_U^*}{\longrightarrow} \GW^{[r]}(U,j^*L).$$
\end{lemma}

\begin{proof}
Denote by $j'_U: U=X'-Y' \to X'$ the open embedding of the complement of the exceptional divisor $Y'$ in the blow-up $X'$ of $X$ along $Y$.
By Lemma \ref{lem:Cod1Uzero}, the functor $(j'_U)^*$ sends the category $\ker Rp_*$, which is generated by $Rj_*\ker Rq_*$, to acyclic objects in $\sPerf{U}$.
In particular, the following diagram commutes
$$\xymatrix{\sPerf(X) \ar[r]^{p^*}_{\sim} \ar[rd]_{j_U^*} & \frac{\sPerf(X')}{\ker Rp_*} \ar[d]^{(j'_U)^*} \\
 & \sPerf(U).}$$
Thus, it suffices to show that the composition
$$\xymatrix{\frac{\sPerf(Y')}{\ker Rq_*} \ar[r]^{Rj_*} & \frac{\sPerf(X')}{\ker Rp_*} \ar[r]^{(j'_U)^*} & \sPerf(U)}$$
has image in the acyclic complexes of $\sPerf(U)$.
But this is a consequence of Lemma \ref{lem:Cod1Uzero}.
\end{proof}

\begin{remark}
The push forward constructed here also yields push-forward for $\K$-theory and its connected cover $K$.
By \cite[Theorem 8.14]{schlichting17}, $GW$ is the homotopy pull back of $\GW \to \K^{hC_2} \leftarrow K^{hC_2}$.
Hence, the push-forwards constructed for $\GW$-theory (and $K$-theory) above also yield push forward maps for $GW$:
$$Ri_*: GW^{[r-c]}(Y,i^*L\cdot \omega_i) \longrightarrow GW^{[r]}(X,L).$$
\end{remark}

\begin{remark}
\label{rmk:FunctorialSquarePushGW}
In the situation of the blow-up square (\ref{eqn:BlowUpSquare}), we have $\omega_p=O_{X'}(-c+1)$, and we define the push-forward map $Rp_*:\GW^{[r]}(X',\omega_p\cdot p^*L) \to \GW^{[r]}(X,L)$ as the result in $GW$-theory of the composition
$$\xymatrix{\left(\sPerf(X'),\omega_p\right) \ar[r]^{\text{Loc}} & \left(\frac{\sPerf(X')}{\ker Rp_*},\omega_p\right) & (\sPerf(X),O_X). \ar[l]_{(p^*,\gamma)}^{\sim}}$$
Then we have functoriality of push forward for the square (\ref{eqn:BlowUpSquare}), $Rp_*\circ Rj_* \circ(1,\can)= Ri_*\circ Rq_*$ as maps $\GW^{[r-1]}(Y',\omega_q\cdot q^*\omega_iL) \to \GW^{[r]}(X,L)$, where $\can:\omega_q\cdot q^*\omega_i \cong \omega_j\cdot j^*\omega_p$ was defined after Lemma \ref{lem:DualIdentnu}.
Similarly, for $GW$ in place of $\GW$.
\end{remark}

\subsection{Strictly perfect complexes on dg schemes}
\label{sec:sPerfDG}

We finish the appendix with details regarding the equivalence (\ref{eqn:KPushCod1}).
\vspace{1ex}

Let $A$ be a differential graded algebra (a.k.a. dga) and $M$ a dg $A$-module. 
All our modules are left modules.
Recall that $M$ is called {\em free} ({\em finitely generated free}) if $M$ is a (finite) direct sum $\bigoplus_{i}A[r_i]$ of modules of the form $A[r_i]$.
A dg $A$-module $M$ is called {\em cellular} ({\em finite cellular}) if $M$ has a (finite) filtration $0 =M_0 \subset M_1 \subset \cdots \subset M$ by dg $A$-submodules such that $M=\bigcup_{i\geq 0}M_i$ and $M_{i}/M_{i-1}$ is (finitely generated) free.
The category of finite cellular modules is closed under extensions in the category of all dg $A$-modules.
See \cite[\S III]{KrizMay}.
\vspace{1ex}

A {\em dg scheme} is a pair $(X,\A)$ of a scheme $X$ together with a quasi-coherent sheaf of dg $O_X$-algebras $\A$.
It is called {\em divisorial} if the scheme $X$ is quasi-compact with ample family of line bundles $\calL = \{L_1,...,L_n\}$  \cite[Definition 2.1.1]{TT}.
A quasi-coherent dg $\A$-module is a complex of quasi-coherent $O_X$-modules $M$ together with a unital and associative multiplication $\A \otimes M \to M$ of complexes.
A homomorphism $M \to N$ of dg $\A$-modules is called {\em quasi-isomorphism} if for all affine open $U \subset X$, the map $M(U) \to N(U)$ is a quasi-isomorphism.

A dg $\A$-module $M$ is called {\em locally cellular} ({\em strictly perfect}) if $X$ has an open affine cover $\bigcup_iU_i = X$, such that the dg $\A(U)$-modules $M(U_i)$ are cellular (finite cellular).
The category $\sPerf(\A)$ of strictly perfect $\A$-modules is closed under extensions in the category $\Qcoh(\A)$ of all quasi-coherent dg $\A$-modules. 
Examples of locally cellular $\A$-modules are the modules of the form $\bigoplus_i\A \otimes L_i[r_i]$ with $L_i \in \calL$ and $r_i\in \Z$, which we call {\em free relative to $\calL$}, and modules that have an exhaustive filtration with successive quotients free relative to $\calL$, which we call {\em cellular relative to $\calL$}.

Tensor product over $\Z$ with bounded complexes of finitely generated free $\Z$-modules makes the categories $\sPerf(\A)$ and $\Qcoh(\A)$ into complicial exact category with weak equivalences the quasi-isomorphisms \cite[A.2.15]{mySedano}.
We write $\calK\sPerf(\A)$ (resp. $\calK(\A)$) for the homotopy category of strictly perfect (resp quasi-coherent) dg $\A$-modules: two maps $M \to N$ are homotopic if their difference is homotopic to zero where the latter means that it factors through the cone $CM=M\otimes_{\Z}C$ of $M$; see \cite[\S\S 3.2.17, A2.14, A2.15]{mySedano}.
The homotopy categories are triangulated categories \cite[3.2.5]{mySedano}.
We write $\D\sPerf(\A)= \sPerf(\A)[\quis^{-1}]$ for the derived category of strictly perfect $\A$-modules, and  $\D(\A) = \Qcoh\A[\quis^{-1}]$ for the derived category of all quasi-coherent dg $\A$-modules.
They are obtained from  $\calK\sPerf(\A)$ and $\calK(\A)$ by inverting the quasi-isomorphisms via a calculus of fractions.

\begin{proposition}
\label{prop:CellApprox}
Let $(X,\A)$ be a divisorial dg scheme with ample family of line bundles $\calL = \{L_i|\ i=1,...,n\}$.
\begin{enumerate}
\item
\label{prop:CellApprox:item1}
For any quasi-coherent dg $\A$-module $M$, there is a map of dg $\A$-modules $F \to M$ which is surjective on homology sheaves, and $F$ is free relative to $\calL$.
\item
\label{prop:CellApprox:item2}
For any quasi-coherent dg $\A$-module $M$, there is a quasi-isomorphism $F \stackrel{\sim}{\longrightarrow} M$ where $F$ is cellular relative to $\calL$.
\item
\label{prop:CellApprox:item3}
For any quasi-isomorphism $M \stackrel{\sim}{\longrightarrow} E$ of quasi-coherent dg $\A$-modules with $E$ strictly perfect, there is a quasi-isomorphism $E'  \stackrel{\sim}{\longrightarrow} M$ with $E'$ strictly perfect.
\item
\label{prop:CellApprox:item4}
Strictly perfect $\A$-modules are compact in $\D(\A)$, 
the category $D(\A)$ is compactly generated with compact generators $\A\otimes L$, $L \in \calL$, and the inclusion of strictly perfect $\A$-modules in all quasi-coherent dg $\A$-modules induces an equivalence 
$$\widetilde{\D}\sPerf(\A) \stackrel{\simeq}{\longrightarrow} \D^c(\A)$$
between the idempotent completion of $\D\sPerf(\A)$ and the subcategory $\D^c(\A)$ of compact objects in $\D(\A)$.
\end{enumerate}
\end{proposition}

\begin{proof}
(\ref{prop:CellApprox:item1}).
By ampleness of $\calL$, for all $r\in \Z$ there is a surjection of sheaves $\bigoplus_{i\in S_r}L_{i,r} \to Z_rM$ onto the sheaf of $r$-cycles of $M$ where $L_{i,r} \in \calL$.
That surjection induces the desired map of dg $\A$-modules $F=\bigoplus_{r\in \Z,\ i\in S_r}\A\otimes L_{i,r}[r] \to M$.

(\ref{prop:CellApprox:item2}).
We construct an increasing sequence of dg $\A$ modules $0=F_0 \subset F_1 \subset \cdots $ together with  maps $\alpha_r:F_r\rightarrow M$ of dg $\A$-modules such that, ${\alpha_{r+1}}_{|F_{r}}=\alpha_r$, the maps $\alpha_r:F_r \to M$ are surjective on homology sheaves for $r\geq 1$, the composition  $\ker \calH_*(\alpha_r) \subset \calH_*(F_{r}) \to \calH_*(F_{r+1})$ is zero, and the quotients $F_{r+1}/F_r$ are free relative to $\calL$. 
Then $F = \bigcup_{i\geq 0}F_r$ is cellular relative to $\calL$, and $F \to M$ is a quasi-isomorphism.
The dg $\A$-module $F_1$ was constructed in (\ref{prop:CellApprox:item1}).
Assume $\alpha_r: F_r \to M$ constructed. 
Using (\ref{prop:CellApprox:item1}), we can choose a map $N \to P(\alpha_r) = PM\times_{M} F_r$ that is surjective on homology sheaves with $N$ free relative to $\calL$ where $PM \to M$ is the path surjection \cite[\S A.2.15]{mySedano}.
Let $F_{r+1}$ be the pushout of $N\to F_r$ along the inclusion $N \subset CN$ into the cone.
Then $F_r \subset F_{r+1}$ with quotient $N[1]$ free relative to $\calL$.
Since the composition $N \to P(\alpha_r) \to PM$ has contractible target, it factors through the cone $CN$ and induces the desired map $\alpha_{r+1}:F_{r+1} \to M$.

(\ref{prop:CellApprox:item3}).
The proof is {\em mutatis mutandis} the same as that of \cite[Lemma 14]{myMV} using the fact that when $X$ is affine, any quasi-isomorphism to a finite cell complex has a section, up to homotopy \cite[Theorem III.2.3]{KrizMay}.

(\ref{prop:CellApprox:item4}).
Strictly perfect dg $\A$-modules are clearly compact in $\calK(\A)$.
It follows from (\ref{prop:CellApprox:item3}) that they are also compact in $\D(\A)$.
It also follows from (\ref{prop:CellApprox:item3}) that $\D\sPerf(\A) \to \D(\A)$ is fully faithful.
By (\ref{prop:CellApprox:item2}), $\D(\A)$ is generated by $\A\otimes L$ with $L\in \calL$.
Since $\A\otimes L \in \D\sPerf(\A)$,  the category of compact objects in $\D(\A)$ is the idempotent completion of $\D\sPerf(\A)$, by \cite[Theorem 2.1]{Neeman:GrothendieckDual}.
\end{proof}

Let $X$ be a divisorial scheme and let $f: \A \to \B$ be a homomorphism of quasi-coherent dg $O_X$-algebras.
Since restriction of scalars from $\B$ to $\A$ preserves quasi-isomorphisms, it induces a map on derived categories $f_*:\D(\B) \to \D(\A)$.
In view of Proposition \ref{prop:CellApprox} (\ref{prop:CellApprox:item2}) its left adjoint $Lf^*: \D(\A) \to \D(\B)$ can be computed as $Lf^*(M) = \B \otimes_{\A}F$ where $F \to M$ is a quasi-isomorphism with $F$ locally cellular.

\begin{corollary}
\label{cor:sPerfEq}
Let $X$ be a divisorial scheme, and let $f: \A \to \B$ be a quasi-isomorphism of  quasi-coherent dg $O_X$-algebras.
Then $Lf^*: \D(\A) \to \D(\B)$ is an equivalence with inverse $f_*$.
In particular, $f^*$ induces an equivalence of the derived categories of strictly perfect complexes up to idempotent completion:
$$f^*: \widetilde{\D}(\sPerf\A)  \stackrel{\simeq}{\longrightarrow}  \widetilde{\D}(\sPerf\B).$$
\end{corollary}

\begin{proof}
Unit and counit of adjunction are quasi-isomorphisms because this is the case locally \cite[Proposition III.4.2]{KrizMay}.
\end{proof}

\section{Explicit comparisons with known results}\label{App:ExplicitComparison}
From Theorem \ref{thm:split}, we obtain 
\[
 \mathbb{G}W^{[r]}(\Gr^{k,n}_{S},(\det\mathcal{Q})^{\otimes (n-k-1)}) \cong  \mathbb{G}W^{[r-(n-k)]}(\Gr^{k-1,n-1}_{S},(\det\mathcal{Q})^{\otimes (n-k)})\oplus \mathbb{G}W^{[r]}(\Gr^{k,n-1}_{S},(\det\mathcal{Q})^{\otimes (n-k-1)}).
\]
To be fully explicit, we use Theorem \ref{thm:compute} to compute the terms on the right-hand side of the above isomorphism. By inspection of Theorem \ref{thm:compute}, we need to consider $k$ up to mod $4$ and $(n-k)$ up to mod $2$. 

Considering each case individually, the full computation reads   
\[ 
\mathbb{G}W^{[r]}(\Gr^{k,n}_{S},(\det\mathcal{Q})^{\otimes (n-k)}) \cong 
\begin{cases}
\mathbb{K}(S)^{\oplus \frac{1}{2}R(k,n-k)} & k(n-k) \,\, \text{odd,} \\
\mathbb{G}W^{[r - 2]}(S)^{\oplus S(k,n-k)} \oplus \mathbb{K}(S)^{\oplus \frac{1}{2}A(k,n-k)} & (n-k) \,\,  \text{odd and } k\equiv 2 (\text{mod }4),\\
\mathbb{G}W^{[r]}(S)^{\oplus S(k,n-k)} \oplus \mathbb{K}(S)^{\oplus \frac{1}{2}A(k,n-k)} &  \text{otherwise,} \\
\end{cases}
 \]
{\small
\[
\mathbb{G}W^{[r]}(\Gr^{k,n}_{S},(\det\mathcal{Q})^{\otimes (n-k-1)}) \cong 
\begin{cases}
    \mathbb{G}W^{[r-(n-k)]}(S)^{\oplus S(k-1,n-k)}\oplus \mathbb{G}W^{[r]}(S)^{\oplus S(k,n-k-1)}\oplus \mathbb{K}(S)^{\oplus \frac{1}{2}(A(k-1,n-k) + A(k,n-k-1))} \\
    \text{for $k \equiv 0 (\text{mod }4)$, $(n-k) \,\, \text{even}$ or $k \equiv 1 (\text{mod }4)$, $(n-k) \,\, \text{odd}$ },  
    \\\\
    \mathbb{G}W^{[r]}(S)^{\oplus S(k,n-k-1)}\oplus  \mathbb{K}(S)^{\oplus \frac{1}{2}(R(k-1,n-k) + A(k,n-k-1))} \\
    \text{for $k$ even, $(n-k) \,\, \text{odd}$},
    \\\\
    \mathbb{G}W^{[r-(n-k)]}(S)^{\oplus S(k-1,n-k)}\oplus  \mathbb{K}(S)^{\oplus \frac{1}{2}(R(k,n-k-1) + A(k-1,n-k))} \\
    \text{for $k$ odd, $(n-k) \,\, \text{even}$},
    \\\\
     \mathbb{G}W^{[r-(n-k)]}(S)^{\oplus S(k-1,n-k)}\oplus \mathbb{G}W^{[r-2]}(S)^{\oplus S(k,n-k-1)}\oplus \mathbb{K}(S)^{\oplus \frac{1}{2}(A(k-1,n-k) + A(k,n-k-1))} \\
    \text{for $k \equiv 2 (\text{mod }4)$, $(n-k) \,\, \text{even}$}, 
     \\\\
     \mathbb{G}W^{[r-(n-k)-2]}(S)^{\oplus S(k-1,n-k)}\oplus \mathbb{G}W^{[r]}(S)^{\oplus S(k,n-k-1)}\oplus \mathbb{K}(S)^{\oplus \frac{1}{2}(A(k-1,n-k) + A(k,n-k-1))} \\
    \text{for $k \equiv 3 (\text{mod }4)$, $(n-k) \,\, \text{odd}$}.
\end{cases}
\]
}

Similarly, for $\mathbb{L}$-theory, we have
 \[
   \mathbb{L}^{[r]}(\Gr^{k,n}_{S},(\det\mathcal{Q})^{\otimes (n-k)}) \cong 
\begin{cases}
0 & k(n-k) \,\, \text{odd,} \\
\mathbb{L}^{[r - 2]}(S)^{\oplus S(k,n-k)} & (n-k) \,\,  \text{odd and } k\equiv 2 (\text{mod }4),\\
\mathbb{L}^{[r]}(S)^{\oplus S(k,n-k)}  &  \text{otherwise,} \\
\end{cases}
 \]
\[
\mathbb{L}^{[r]}(\Gr^{k,n}_{S},(\det\mathcal{Q})^{\otimes (n-k-1)}) \cong 
\begin{cases}
    \mathbb{L}^{[r-(n-k)]}(S)^{\oplus S(k-1,n-k)}\oplus \mathbb{L}^{[r]}(S)^{\oplus S(k,n-k-1)}\\
    \text{for $k \equiv 0 (\text{mod }4)$, $(n-k) \,\, \text{even}$ or $k \equiv 1 (\text{mod }4)$, $(n-k) \,\, \text{odd}$ },  
    \\\\
    \mathbb{L}^{[r]}(S)^{\oplus S(k,n-k-1)} \\
    \text{for $k \,\, \text{even}$, $(n-k) \,\, \text{odd}$  },
    \\\\
    \mathbb{L}^{[r-(n-k)]}(S)^{\oplus S(k-1,n-k)} \\
    \text{for $k \,\, \text{odd}$, $(n-k) \,\, \text{even}$  },
    \\\\
     \mathbb{L}^{[r-(n-k)]}(S)^{\oplus S(k-1,n-k)}\oplus \mathbb{L}^{[r-2]}(S)^{\oplus S(k,n-k-1)}\\
    \text{for $k \equiv 2 (\text{mod }4)$, $(n-k) \,\, \text{even}$}, 
     \\\\
     \mathbb{L}^{[r-(n-k)-2]}(S)^{\oplus S(k-1,n-k)}\oplus \mathbb{L}^{[r]}(S)^{\oplus S(k,n-k-1)}\\
    \text{for $k \equiv 3 (\text{mod }4)$, $(n-k) \,\, \text{odd}$}.
\end{cases}
\]

In the following, we check that our results coincide with existing computations known in the literature.    

\subsection{Explicit computation of Witt Groups and Comparisons}
Using the fact that for all $r, i\in \mathbb{Z},L^{[r]}_{i}(S) \cong W^{r-i}(S)$ \cite[Sections 7]{schlichting17}, Balmer's 4-periodic triangular Witt groups are given as follows: 

 \[
   W^{i}(\Gr^{k,n}_{S},(\det\mathcal{Q})^{\otimes (n-k)}) \cong 
\begin{cases}
0 & k(n-k) \,\, \text{odd,} \\
W^{ i-2}(S)^{\oplus S(k,n-k)} & (n-k) \,\,  \text{odd and } k\equiv 2 (\text{mod }4),\\
W^{i}(S)^{\oplus S(k,n-k)}  &  \text{otherwise } \\
\end{cases}
 \]

\[
W^{i}(\Gr^{k,n}_{S},(\det\mathcal{Q})^{\otimes (n-k-1)}) \cong 
\begin{cases}
    W^{i-(n-k)}(S)^{\oplus S(k-1,n-k)}\oplus W^{i}(S)^{\oplus S(k,n-k-1)}\\
    \text{for $k \equiv 0 (\text{mod }4)$, $(n-k) \,\, \text{even}$ or $k \equiv 1 (\text{mod }4)$, $(n-k) \,\, \text{odd}$ },  
    \\\\
    W^{i}(S)^{\oplus S(k,n-k-1)} \\
    \text{for $k \,\, \text{even}$, $(n-k) \,\, \text{odd}$  },
    \\\\
    W^{i-(n-k)}(S)^{\oplus S(k-1,n-k)} \\
    \text{for $k \,\, \text{odd}$, $(n-k) \,\, \text{even}$  },
    \\\\
     W^{i-(n-k)}(S)^{\oplus S(k-1,n-k)}\oplus W^{i-2}(S)^{\oplus S(k,n-k-1)}\\
    \text{for $k \equiv 2 (\text{mod }4)$, $(n-k) \,\, \text{even}$}, 
     \\\\
     W^{i-(n-k)-2}(S)^{\oplus S(k-1,n-k)}\oplus W^{i}(S)^{\oplus S(k,n-k-1)}\\
    \text{for $k \equiv 3 (\text{mod }4)$, $(n-k) \,\, \text{odd}$}.
\end{cases}
\]
\subsubsection{Comparison with \cite{zibrowius2011witt}}
In particular, we see that our results agree with \cite[Section 4.3]{zibrowius2011witt} in the case $S = \mathbb{C}$. Here, we use the identities $S(k,l) = S(k,l-1)$ when $l$ is odd and $S(k-1,l) + S(k,l-1) = S(k,l)$ when $k$ and $l$ are both even, $W^{i}(\mathbb{F}) = 0 $ for all $i \neq 0 \mod{4}$ for a field $\mathbb{F}$ for which $2$ is invertible \cite[Theorem 5.6]{balmer2001triangular}, and the fact that $W^{0}(\mathbb{C}) = \mathbb{Z}/2\mathbb{Z}$.

\subsubsection{Comparison with \cite{balmer2012witt}}
Let $X$ be a regular, noetherian and separated scheme over $\mathbb{Z}[1/2]$. Let $K$ be a line bundle on $X$, $\mathcal{V}$ a vector bundle of rank $d+e$ on $X$, $\pi: \Gr_{X}(d,\mathcal{V}) \rightarrow X$ the structure morphism associated to the Grassmannian of $d$-dimensional subbundles of $\mathcal{V}$ and $\mathcal{T}_{d}$ the tautological bundle of the Grassmannian. Then, an \textit{explicit} isomorphism of $W^{0}(X,\mathcal{O}_{X})$ modules is constructed in \cite[Corollary 7.1]{balmer2012witt}:
\[
 W^{k}\left( \Gr_{X}(d,\mathcal{V}), \pi^{*}(K)\otimes \det(\mathcal{T}_{d})^{\otimes l}\right) \cong \bigoplus_{\substack{\Lambda \,\, \text{even} \,\, \text{s.t} \\ t(\Lambda) = l \in \mathbb{Z}/2\mathbb{Z}}}W^{k-|\Lambda|}\left( X, K \otimes \det(\mathcal{V})^{\otimes-\rho(\Lambda)}\right).
\]
Here, a Young diagram $\Lambda$ contained in a $(d\times e)$-rectangle is called \textit{even} if all the segments of its boundary that are strictly inside the frame have \textit{even length}. The element $t(\Lambda) \in \mathbb{Z}/2\mathbb{Z}$ denotes the class of \textit{half the perimeter} of the Young diagram $\Lambda$ (note, the perimeter of any Young diagram is always even because, for example, it is equal to the perimeter of the smallest rectangle bounding it). The element $\rho(\Lambda) \in \{0,\dots,d\}$ denotes the number of non-zero rows of $\Lambda$.




In particular, for $X = S$ a divisorial scheme, $K = \mathcal{O}_{X}$ and $\mathcal{V} = \mathcal{O}_{X}^{d+e}$, we have an isomorphism 
\begin{align} \label{iso:bcwittgroupfield}
W^{k}\left( \Gr_{S}^{d,d+e}, \det(\mathcal{Q})^{\otimes l}\right)  \cong \bigoplus_{\substack{\Lambda \,\, \text{even} \,\, \text{s.t} \\ t(\Lambda) = l \in \mathbb{Z}/2\mathbb{Z}}}W^{k-|\Lambda|}\left( S\right).
\end{align}
To check that our results coincide, we need to enumerate the even Young diagrams in a $d\times e$ frame. Fortunately, this is done in \cite[Corollary 7.2]{balmer2012witt}. 

Let `4-block' mean `$2\times 2$' square. Then, every even Young diagram is either 
\begin{enumerate}[(a)]
    \item \label{item:even1} A union of 4-blocks,
    \item \label{item:even2} A single row plus 4-blocks, in which case $e$ is necessarily even,
    \item \label{item:even3} A single column plus 4-blocks, in which case $d$ is necessarily even, 
    \item \label{item:even4} A single row and a single column plus 4-blocks, in which case $d$ and $e$ are necessarily odd. 
\end{enumerate}
Note that these possibilities are exclusive and can be enumerated by counting the Young diagrams built by using 4-blocks that fit within the relevant frames. 

We will compute $W^{i}(\Gr_{S}^{d,d+e}, \det(\mathcal{Q})^{\otimes e})$ and $W^{i}(\Gr_{S}^{d,d+e}, \det(\mathcal{Q})^{\otimes e-1})$ using isomorphism (\ref{iso:bcwittgroupfield}) to check that our formulae coincide.

\paragraph{\textit{Computation of $W^{i}(\Gr_{S}^{d,d+e}, \det(\mathcal{Q})^{\otimes e})$} }
\begin{enumerate}
    \item Suppose $de$ is odd. Therefore, $d$ and $e$ are necessarily odd. Thus, we need to count all even Young diagrams $\Lambda$ in the frame $(d\times e)$ such that $t(\Lambda)$ is odd. Using the above characterization, we see that there are no such even diagrams. 
    \item Suppose $e$ is odd and $d \equiv_{4}2$. We need to count all even Young diagrams $\Lambda$ in the frame $(d\times e)$ such that $t(\Lambda)$ is odd. By inspection, we see that we are necessarily in case (\ref{item:even3}) of the above characterization. The number of such Young diagrams are $\binom{\frac{d}{2}+\frac{e-1}{2}}{\frac{d}{2}} = S(d,e)$. Moreover, given such a Young diagram $\Lambda$, we compute $|\Lambda| \equiv_{4} d \equiv_{4} 2$, as the 4-blocks do not contribute to the modulus mod 4. Thus, our computations agree in this case.  
    \item For the case $e$ is odd and $d \equiv_{4}0$, similar reasoning to the above holds, but this time $|\Lambda| \equiv_{4} d \equiv_{4} 0$. For the case $e$ is even, $t(\Lambda)$ is even. As Young diagrams $\Lambda$ in cases (\ref{item:even2}) and (\ref{item:even3}) have $t(\Lambda)$ necessarily odd, we are in case (\ref{item:even1}). The number of Young diagrams in this case is precisely $S(d,e)$, and have a modulus a multiple of 4. 
\end{enumerate}
\paragraph{\textit{Computation of $W^{i}(\Gr_{S}^{d,d+e}, \det(\mathcal{Q})^{\otimes e-1})$} }
\begin{enumerate}
    \item Suppose $e$ is even and $d \equiv_{4} 0$. We need to count all even Young diagrams $\Lambda$ in the frame $(d\times e)$ such that $t(\Lambda)$ is odd. Therefore, we are necessarily in cases (\ref{item:even2}) and (\ref{item:even3}). The number of Young diagrams $\Lambda$ in case (\ref{item:even2}) is $S(d-1,e)$, with modulus $|\Lambda | \equiv_{4} e$. The number of Young diagrams $\Lambda$ in case $(\ref{item:even3})$ is $S(d-1,e)$, with modulus $|\Lambda | \equiv_{4} d \equiv_{4} = 0$. Now suppose $d \equiv_{4} 1$ and $e$ is odd. We need to count all even Young diagrams $\Lambda$ in the frame $(d\times e)$ such that $t(\Lambda)$ is even. Therefore, we are necessarily in cases (\ref{item:even1}) and (\ref{item:even4}). The number of Young diagrams $\Lambda$ in case (\ref{item:even1}) is $S(d,e-1)$, with modulus $|\Lambda | \equiv_{4} 0$. The number of Young diagrams $\Lambda$ in case $(\ref{item:even4})$ is $S(d-1,e)$, with modulus $|\Lambda | \equiv_{4} d +e -1 \equiv_{4} = e$.
    \item Suppose $d$ is even and $e$ is odd.  We need to count all even Young diagrams $\Lambda$ in the frame $(d\times e)$ such that $t(\Lambda)$ is even. Therefore, we are necessarily in case(\ref{item:even1}). The number of Young diagrams $\Lambda$ in case (\ref{item:even1}) is $S(d,e-1)$, with modulus $|\Lambda | \equiv_{4} 0$.
    \item Suppose $d$ is odd and $e$ is even. We need to count all even Young diagrams $\Lambda$ in the frame $(d\times e)$ such that $t(\Lambda)$ is odd. Therefore, we are necessarily in case (\ref{item:even2}). The number of Young diagrams $\Lambda$ in case (\ref{item:even2}) is $S(d-1,e)$, with modulus $|\Lambda | \equiv_{4} e$.
    \item Suppose $d \equiv_{4} 2$ and $e$ is even. We need to count all even Young diagrams $\Lambda$ in the frame $(d\times e)$ such that $t(\Lambda)$ is odd. Therefore, we are necessarily in cases (\ref{item:even2}) and (\ref{item:even3}). The number of Young diagrams $\Lambda$ in case (\ref{item:even2}) is $S(d-1,e)$, with modulus $|\Lambda | \equiv_{4} e$. The number of Young diagrams $\Lambda$ in case (\ref{item:even3}) is $S(d,e-1)$, with modulus $|\Lambda | \equiv_{4} d \equiv_{4} 2$.
    \item Suppose $d \equiv_{4} 3$ and $e$ is odd. We need to count all even Young diagrams $\Lambda$ in the frame $(d\times e)$ such that $t(\Lambda)$ is even. Therefore, we are necessarily in cases (\ref{item:even1}) and (\ref{item:even4}). The number of Young diagrams $\Lambda$ in case (\ref{item:even1}) is $S(d,e-1)$, with modulus $|\Lambda | \equiv_{4} 0$. The number of Young diagrams $\Lambda$ in case (\ref{item:even4}) is $S(d-1,e)$, with modulus $|\Lambda | \equiv_{4} d +e -1\equiv_{4} e+2$.
\end{enumerate}

Thus, our computation coincides with \cite{balmer2012witt}. 

\subsection{Comparisons of Grothendieck--Witt Groups}

\subsubsection{Comparison with \cite{karoubi2021grothendieck} and \cite{rohrbach2022projective}}
Our computation of the Hermitian K-theory of Grassmannians in particular generalises the analogous computation for projective space given in \cite{karoubi2021grothendieck} and \cite{rohrbach2022projective} (the latter computing the Hermitian K-theory of projective bundles). 

Both compute that, for a scheme $X$ over $\mathbb{Z}[1/2]$ with an ample family of line bundles, 
\[
GW^{[r]}(\mathbb{P}^{n}_{X}, \mathcal{O}(i)) \cong 
\begin{cases}
    GW^{[r]}(X)\oplus K(X)^{\oplus m} & n = 2m, \,\, i \,\, \text{even}, \\
    GW^{[r]}(X)\oplus K(X)^{\oplus m} \oplus GW^{[r-n]}(X) & n = 2m+1, \,\, i \,\, \text{even}, \\
    K(X)^{\oplus m} & n = 2m-1, \,\, i \,\, \text{odd}, \\
    K(X)^{\oplus m} \oplus GW^{[r-n]}(X) & n = 2m, \,\, i \,\, \text{odd}.
\end{cases}
\]
Setting $k = 1$ (or equivalently $k = n-1$) in our computation of $\mathbb{G}W^{[r]}(\Gr^{k,n}_{S}, \det\mathcal{Q}^{\otimes i})$, we see that our results coincide with the above. 

\subsubsection{Comparison with \cite{rohrbach2021atiyah}}
In \cite[Theorem 10.3.10]{rohrbach2021atiyah}, the Hermitian K-theory spectrum 

$GW^{[r]}(\Gr_{\mathbb{F}}^{d,d+e}, (\det\mathcal{Q})^{e})$ is computed for a field $\mathbb{F}$ of characteristic zero and $e$ \textit{even}. It is computed 
$$GW^{[r]}(\Gr_{\mathbb{F}}^{d,d+e}, (\det\mathcal{Q})^{e}) \cong GW^{[r]}(\mathbb{F})^{S(d,e)} \oplus K(\mathbb{F})^{\oplus\frac{1}{2}\left( R(d,e) - S(d,e)\right)},$$ 
which is precisely what we compute in this case.

\subsubsection{Comparison with \cite{huang2023connecting}}
For a regular scheme $S$, a vector bundle $\mathcal{V}$  
over $S$ of rank $r$ and a line bundle $L$ over $S$, an \textit{explicit equivalence of spectra} is constructed in \cite[Theorem 10.39]{huang2023connecting}:
\[
GW^{[n]}(\Gr_{d}(\mathcal{V}), L \otimes \Delta_{d}^{\otimes l}) \cong  \bigoplus_{\Lambda \in \mathfrak{B}^{l}_{d,m}}K(S) \oplus \bigoplus_{\Pi \in \mathfrak{U}^{l}_{d,m}} GW^{[n - |\Pi|]}(S, L \otimes \det(\mathcal{V})^{\otimes \rho(\Pi)}). 
\]
Here, $\Delta_{d}$ is the determinant of the tautological subbundle of rank $d$, $m \coloneqq r-d$ and $l \in \mathbb{Z}$ is any integer. We will verify our computations coincide in the case $S$ is divisorial, $\mathcal{V} = \mathcal{O}_S^{r}$, $L = \mathcal{O}_S$.

The set $\mathfrak{U}_{d,m}^{l}$ consists of all even Young diagrams $\Pi$ in a $d \times m$ frame such that $t(\Pi) \equiv_{2} l$, where $t(\Pi) \in \mathbb{Z}/2\mathbb{Z}$ is the same element defined by \cite{balmer2012witt}. The element $\rho(\Pi)$ is also the same defined by \cite{balmer2012witt}. The index set $\mathfrak{B}^{l}_{d,m}$ is new and requires some explanation. 

In \cite{huang2023connecting}, the set $\mathfrak{B}_{d,m}$ is defined to be the set of all \textit{buffalo-check} Young diagrams in a $d\times m$ frame. A buffalo-check Young diagram is a pair $(\Lambda,c_{\Lambda})$, where $\Lambda$ is a \textit{K-even} Young diagram and $c_{\Lambda}$ is a \text{center} of K-even Young diagram $\Lambda$. Let us explain these terms. 

Denote by $b(\Lambda)$ the lattice path that goes from lower-left to upper-right corner in the $d\times m$ such that the Young diagram $\Lambda$ lives precisely in the upper-left area above the path, then $b(\Lambda)$ consists of \textit{segments} $s_{1},\dots ,s_{l}$, where we order the segments from left to right. We refer to \cite{huang2023connecting} for the relevant pictures. Then, $\Lambda$ is called \textit{K-even} if 
\begin{enumerate}[(i)]
    \item the first segment $s_{1}$ is vertical \label{item: Keven1}
    \item or, there exists an integer $r(\Lambda) \geq 2$ such that $s_{2}, \dots, s_{r(\Lambda)-1} $ has even length and $s_{r(\Lambda)}$ is vertical of odd length. \label{item:Keven2} 
\end{enumerate}

A \textit{center} for a K-even Young diagram $\Lambda$ is a box in the angle between the segments $s_{1}$ and $s_{2}$ when condition $(\ref{item: Keven1})$ holds, or between $s_{r(\Lambda)}$ and $s_{r(\Lambda)+1}$ when condition $(\ref{item:Keven2})$ holds. In particular, a K-even Young diagram can have at most 2 centers. Again, we refer to \cite{huang2023connecting} for the relevant pictures. 

In particular, the center of a K-even Young diagram \textit{cannot} be located at the intersection of even rows and even columns. This allows \cite{huang2023connecting} to compute 
\[
|\mathfrak{B}_{d,m}| = \binom{d+m}{d} - \binom{\lfloor d/2 \rfloor + \lfloor m/2 \rfloor}{\lfloor d/2 \rfloor } = R(d,m) - S(d,m).
\]
$\left(\text{In this section , we take $R(d,m) \coloneqq \binom{d+m}{d}$ and $S(d,m) \coloneqq \binom{\lfloor d/2 \rfloor + \lfloor m/2 \rfloor}{\lfloor d/2 \rfloor }$}.\right)$
 
 Moreover, the set $\mathfrak{B}_{d,m}$ is partitioned as $\mathfrak{B}_{d,m} = \mathfrak{B}^{d}_{d,m} \bigsqcup \mathfrak{B}^{d+1}_{d,m}$, where the superscript is taken modulo 2. Here, $\mathfrak{B}^{d}_{d,m}$ are those buffalo-check young diagrams centered at a $(i,j)$ box, for $i,j$ \textit{odd}; and $\mathfrak{B}^{d+1}_{d,m}$ are those buffalo-check young diagrams centered at a $(i,j)$ box, for either $i,j$ \textit{even}. 
Following \cite{huang2023connecting} and defining $\beta^{l}_{d,m} \coloneqq | \mathfrak{B}_{d,m}^{l}|$, it is computed 
\[
\beta^{l}_{d,m} = \begin{cases}
    \frac{1}{2} \left( \binom{d+m}{d} -  \binom{\lfloor d/2 \rfloor + \lfloor m/2 \rfloor}{\lfloor d/2 \rfloor }\right) = \frac{1}{2}(R(d,m) - S(d,m)) & \text{if $d$ is even or $m$ is even,} \\
     \frac{1}{2} \binom{d+m}{d} = \frac{1}{2}R(d,m) & \text{if $d,m,l$ are odd,}\\
     \frac{1}{2} \binom{d+m}{d} -  \binom{\lfloor d/2 \rfloor + \lfloor m/2 \rfloor}{\lfloor d/2 \rfloor} = \frac{1}{2}R(d,m) - S(d,m) & \text{if $d,m$ are both odd and $l$ is even.}
\end{cases}
\]
Therefore, to verify that our computation agrees with the computation given in \cite{huang2023connecting}, it suffices to check that the coefficients $\beta^{l}_{d,m}$ coincide with our coefficients for the $\mathbb{K}$-theory terms for the various cases of $d,m$ and $l$.
\paragraph{\textit{Computation of $\mathbb{G}W^{[r]} (\Gr_{S}^{d,d+e}, \det(\mathcal{Q})^{\otimes e})$} }
\begin{enumerate}
    \item Suppose $de$ is odd. Then, $\beta^{e}_{d,e} = \frac{1}{2}R(d,e)$.
    \item Suppose $e$ is odd and $d \equiv_{4}2$. Then, $\beta^{e}_{d,e} = \frac{1}{2}(R(d,e) - S(d,e)) = \frac{1}{2}A(d,e)$.
    \item For the case $e$ is odd and $d \equiv_{4}0$, or the case $e$ is even, it is the same as the previous case.
\end{enumerate}

\paragraph{\textit{Computation of $\mathbb{G}W^{[r]}(\Gr_{S}^{d,d+e}, \det(\mathcal{Q})^{\otimes e-1})$} }
\begin{enumerate}
    \item Suppose $e$ is even and $d \equiv_{4} 0$, or suppose $d \equiv_{4} 1$ and $e$ is odd. Then, in the first case, $\beta^{e-1}_{d,e} = \frac{1}{2}(R(d,e) - S(d,e)) = \frac{1}{2}(A(d-1,e) + A(d,e-1))$. In the second case, $\beta^{e-1}_{d,e} = \frac{1}{2}R(d,e) - S(d,e) = \frac{1}{2}(A(d-1,e) + A(d,e-1))$.
    \item Suppose $d$ is even and $e$ is odd. Then, $\beta^{e-1}_{d,e} = \frac{1}{2}(R(d,e) - S(d,e)) = \frac{1}{2}(R(d-1,e) + A(d,e-1))$.
    \item Suppose $d$ is odd and $e$ is even. Then, $\beta^{e-1}_{d,e} = \frac{1}{2}(R(d,e) - S(d,e)) = \frac{1}{2}(R(d,e-1) + A(d-1,e))$.
    \item Suppose $d \equiv_{4} 2$ and $e$ is even. Then, $\beta^{e-1}_{d,e} = \frac{1}{2}(R(d,e) - S(d,e)) = \frac{1}{2}(A(d-1,e) + A(d,e-1))$.
    \item Suppose $d \equiv_{4} 3$ and $e$ is odd. Then, $\beta^{e-1}_{d,e} = \frac{1}{2}R(d,e) - S(d,e) = \frac{1}{2}(A(d-1,e) + A(d,e-1))$.
\end{enumerate}
Thus, our computation coincides with \cite{huang2023connecting} in this case.

\newpage

\bibliography{grassmannian}
\bibliographystyle{alpha}

\end{document}